\documentclass[preprint]{elsarticle}

\usepackage{amsfonts,amssymb,amsmath}

\usepackage{euscript,amsfonts,amssymb,amsmath,amscd,amsthm}

\usepackage{color}

\usepackage{hyperref}

\newtheorem{theorem}{Theorem}[section]
\newtheorem{lemma}[theorem]{Lemma}
\newtheorem{definition}[theorem]{Definition}
\newtheorem{proposition}[theorem]{Proposition}
\newtheorem{conjecture}[theorem]{Conjecture}
\newtheorem{corollary}[theorem]{Corollary}

\numberwithin{equation}{section}

\newcommand{\ilim}{\mathop{\varinjlim}\limits}

\renewcommand{\H}{{\mathcal H}_q}

\newcommand{\F}{\mathbb F_q}

\newcommand{\B}{\mathbb B}

\newcommand{\U}{\mathbb U}

\newcommand{\GLB}{\mathbb{ GLB}}

\newcommand{\GLU}{\mathbb{ GLU}}

\newcommand{\SLU}{\mathbb{ SLU}}

\newcommand{\Y}{\mathbb Y}

\newcommand{\A}{\mathcal A}

\newcommand{\GL}[1]{\mathbb{GL}(#1,q)}
\newcommand{\SL}[1]{\mathbb{SL}(#1,q)}

\newcommand{\C}{\mathcal C}

\newcommand{\CY}{\C \Y}

\newcommand{\Sym}{\mathfrak S}

\newcommand{\St}{\mathcal St}

\newcommand{\Cyl}{{\rm Cyl}}

\newcommand{\Sp} {\mathcal Sp}

\newcommand{\BI}{\mathbb {BI}}

\newcommand{\Uni}{{\rm \mathbf Uni}}

\begin{document}


\title
{Finite traces and representations of the group of infinite matrices over a finite
field.}


\author{Vadim Gorin}
\address{Massachusetts Institute of Technology, Cambridge, MA, USA and
Institute for Information Transmission Problems of Russian Academy of Sciences, Moscow, Russia}
 \ead{vadicgor@gmail.com}

\author{Sergei Kerov }

\address{(1946--2000)}

\author{Anatoly Vershik}
\address{St.Petersburg Department of
V.~A.~Steklov   Institute   of   Mathematics of    Russian   Academy   of   Sciences, Saint
Petersburg, Russia} \ead{avershik@gmail.com}

\begin{abstract}
The article is devoted to the representation theory of locally compact infinite-dimensional group
$\GLB$ of almost upper-triangular infinite matrices over the finite field with $q$ elements.  This
group was defined by S.K., A.V., and Andrei Zelevinsky in 1982 as an adequate $n=\infty$ analogue
of general linear groups $\GL{n}$. It serves as an alternative to $\GL{\infty}$, whose
representation theory is poor.

Our most important results are the description of semi-finite unipotent traces (characters) of the
group $\GLB$ via certain probability measures on the Borel subgroup $\B$ and the construction of
the corresponding von Neumann factor representations of type $II_\infty$.

As a main tool we use the subalgebra  $\A(\GLB)$ of smooth functions in the group algebra
$L_1(\GLB)$. This subalgebra is an inductive limit of the finite--dimensional group algebras
${\mathbb C}(\GL{n})$ under parabolic embeddings.

As in other examples of the asymptotic representation theory we
 discover remarkable properties of the infinite case which does not take place for finite groups, like multiplicativity of indecomposable characters
 or connections to probabilistic concepts.

 The infinite dimensional Iwahori-Hecke algebra $\H(\infty)$ plays a special role in our
 considerations and allows to understand the deep analogy of the developed theory with the representation theory of infinite symmetric
 group $S(\infty)$
 which had been intensively studied in numerous previous papers.
\end{abstract}

\maketitle

\begin{flushright}
To the memory of Andrei Zelevinsky.
\end{flushright}

\tableofcontents

\section*{Historical preface}

My joint work with S.~Kerov on the asymptotic representation theory of the matrix groups $\GL{n}$
over finite field as the rank $n$ grows to infinity, was started at the beginning of 80s as a
continuation of our papers devoted to analogous problems for symmetric groups of growing ranks at
the end of 70-th. It is a part of what I called ``the asymptotic representation theory''.

 The ``trivial'' embedding $\GL{n}\hookrightarrow\GL{n+1}$ does not lead to an interesting or
useful theory. However, another ``true'' (i.e.\ parabolic) embedding of the group algebras of
$\GL{n}$ was well-known starting from the very first papers on the representation theory of
$\GL{n}$ (see \cite{Green}, \cite{Zel}, \cite{F}, etc). It was used by A.~Zelevinsky and us (see
\cite{James}) to define a natural \emph{limit object} (i.e.\
 inductive limit) which is the group (named $\GLB$) of infinite matrices over finite field with finitely
 many non-zero elements below the main diagonal. The results of the book \cite{Zel} in which the
 representation theory of $\GL{n}$ is studied via the Hopf algebras theory, are not directly
 related to the asymptotic representation theory (because it does not consider inductive limits)
  but are substantial and important for the problems' setup.

However, the project was suspended and we returned to this topic only in the middle of 90s with the
first article \cite{VK98} appearing in 1998. In this paper we gave main definitions and sketched
the plan of further research. In 1998-1999 we prepared some more detailed texts. Already after the
sad death of my former student and coauthor an outstanding mathematician and person S.~Kerov, an
improved version of these texts was published \cite{VK_4d}. All these texts contained (mostly
without proofs) a number of statements forming an initial foundation of the asymptotic
representation theory of group $\GLB$.

 The great contribution of of Sergei Kerov (12.06.1946-30.07.2000), and of Andrey Zelevinsky (30.01.1953 -10.04.2013)
 to the different areas of the representation theory will be appreciated by mathematicians of the next generations.

\smallskip

V.~Gorin following the general plan contained in \cite{VK_4d} interpreted and supplemented with
complete proofs the statements of that article. The present paper is a result of this work. We can
say that this article, finally, concludes the first step in the study of the representation theory
of $\GLB$.

\begin{flushright}
 A.~Vershik
\end{flushright}

\section{Introduction}

\subsection{Overview}

Let $\F$ be the finite field with $q$ elements and let $\GL{n}$ denote the group of all invertible
$n\times n$ matrices over $\F$. In the present article we study the group $\GLB$ of all almost
upper-triangular matrices (i.e.\ containing finitely many nonzero elements below diagonal) over
$\F$.  In other words, if $V_\infty$ is the (countable) vector space of all finite vectors in
$\F^\infty$ and $V_n$ is the subspace spanned by the first $n$ basis vectors, then $\GLB$ consists
of all linear transforms of $V_{\infty}$ preserving all but finitely many spaces $V_n$.

We view group $\GLB$ (and some other closely related groups such as $\GLU$ and $\SLU$ mentioned below) as a natural
\emph{infinite--dimensional analogue} of groups $\GL{n}$, extending and continuing the classical
representation theory of these groups. We explain this point of view in more detail in Section
\ref{Section_comments}.

An important feature of the infinite-dimensional group $\GLB$ is its \emph{local compactness}.
$\GLB$ possesses the Haar measure $\mu_{\GLB}$ which allows to introduce the group algebra
$L_1(\GLB,\mu_{\GLB})$ with multiplication given by the convolution. In our study we intensively
use an important dense subalgebra $\A(\GLB)$ of $L_{1}(\GLB,\mu_\GLB)$ consisting of all
continuous functions with compact support taking only finitely many values.

As opposed to semisimple Lie groups over arbitrary fields \cite{Ber} the group $\GLB$ \emph{is not}
a type $I$ group. While $\GLB$ \emph{has no non-trivial finite characters}, but is has a rich
family of \emph{semi--finite traces} (characters), which might be singular on the group itself, but
are finite on $\A(\GLB)$. These traces give rise to type $II$ \emph{factor representations} and
form a basis for the representation theory and harmonic analysis on $\GLB$. They are the main
object of our interest.

In the present article  we concentrate on the following topics, where we prove a variety of
results.
\begin{itemize}
\item We study the algebra $L_1(\GLB,\mu_\GLB)$ and its dense
locally semisimple subalgebra $\A(\GLB)$. The structure of the latter as an inductive limit of
finite-dimensional algebras is explained.

\item We show that $\A(\GLB)$ is a direct sum of certain ideals. The ideals are not isomorphic, but each of them is
Morita equivalent to the \emph{infinite-dimensional Iwahori-Hecke algebra} $\H(\infty)$ (which is
the inductive limit of finite--dimensional $\H(n)$). In our setting this means that the traces of
the ideals are \emph{canonically identified} with those of $\H(\infty)$. As a consequence we
describe \emph{all finite traces} of $\A(\GLB)$.

\item We distinguish and explore the remarkable properties of the family of unipotent (principal) traces,
which are most closely related infinite-dimensional Iwahori-Hecke algebra $\H(\infty)$, infinite
symmetric group $S(\infty)$ and unipotent representations of $\GL{n}$.
\item We show that each unipotent trace can be identified with conjugation--invariant ergodic
\emph{probability measure on the Borel subgroup} $\B\subset\GLB$ of all upper-triangular matrices.
A number of theorems and conjectures on the structure of such measures is presented.
\item We give a construction for the wide family of representations of $\GLB$ based on the natural
action of $\GLB$ in the spaces of flags in the infinite-dimensional vector space over $\F$ and in
the \emph{principal grouppoid} defined by this action. We prove that these representations are von
Neumann type $II_\infty$ factor representations and compute their traces, which are identified
with extreme (indecomposable) unipotent traces of $\A(\GLB)$ with explicit parameters in our
classification of all finite traces.
\item We study the (bi-)--regular representation of $\GLB$, show that is possesses a
natural trace which is finite on $\A(\GLB)$ and decompose this trace into extremes.
\end{itemize}

In Sections \ref{Section_intro_contents}--\ref{section_intro_before_comments} a more detailed
description of our work is given. In Section \ref{Section_comments} we motivate our definitions
and some of the choices, which otherwise might seem arbitrary. We also summarize the similarities
with asymptotic representation theory of symmetric groups in the same section. The brief list of
key notations and theorems is given in Section \ref{Section_intro_notations}. Finally, we want to
remark that most of the results of the present paper were announced (without proofs) in articles
\cite{VK98}, \cite{V_lectures}, \cite{VK_4d}. In this text we refine those statements and give
full proofs of them, also a number of entirely new results is presented.

\subsection{Schwartz-Bruhat algebra $\A(\GLB)$ and its traces}

\label{Section_intro_contents}

One of the central objects of our study is algebra $\A(\GLB)\subset L_{1}(\GLB,\mu_\GLB)$
consisting of all continuous functions on $\GLB$ with compact support taking only finitely many
values. $\A(\GLB)$ can be viewed as an analogue of the algebra of smooth functions or
Schwartz-Bruhat algebra in the theory of linear $p$-adic groups, see e.g.\ \cite{BZ}, \cite{GGP}.

In Section \ref{Section_AG_as_semisimple} we show that algebra $\A(\GLB)$ is an inductive limit of
group algebras $\A(\GLB)_n\simeq\mathbb C(\GL{n}))$. However, the arising embeddings $i_n:\mathbb
C(\GL{n}) \hookrightarrow \mathbb C(\GL{n+1})$ are not induced by group embeddings, instead we
should use \emph{parabolic emebeddings}, which are averagings by certain subgroups. This
description implies that $\A(\GLB)$ is a \emph{locally semisimple algebra}, which means, in
particular, that the enveloping $C^*$--algebra of $\GLB$ is an approximately finite dimensional
{(AF)} algebra; see e.g.\ \cite{Br}, \cite{VK_AF}, \cite{SV}, \cite{Kerov_book}.

As for every locally semisimple algebra the structure of the algebra $\A(\GLB)$ is uniquely
defined by its \emph{Bratteli diagram}. Recall that Bratteli diagram is a graded graph with
vertices at level $n$ symbolizing irreducible representations of $\A(\GLB)_n$ and edges between
adjacent levels symbolizing the inclusion relations of the representations. The inclusions $i_n:
\A(\GLB)_n \hookrightarrow \A(\GLB)_{n+1}$ are not unital and the algebra $\A(\GLB)$ has no unit
element, thus, each vertex of the Bratteli diagram is to be supplemented with additional label
which is the dimension of the corresponding irreducible representation. In our case these numbers
are dimensions of irreducible representations of groups $\GL{n}$ and they admit explicit formulas
(e.g.\ a $q$--analogue of the classical hook formula, see e.g.\ \cite[Chapter IV]{M})

We show in Section \ref{Section_AG_as_semisimple} that the Bratteli diagram of $\A(\GLB)$ is a
\emph{disjoint union} of countably many copies of the \emph{Young graph} $\Y$ with shifted grading.
Recall that level $n$ of $\Y$ consist of all Young diagrams with $n$ boxes (equivalently,
partitions of $n$) with edges joining the diagrams which differ by addition of a single box.
Therefore, the algebra $\A(\GLB)$ is a direct sum of countably many (non-isomorphic) ideals
corresponding to different copies of $\Y$.

We also show that infinite \emph{Iwahori-Hecke algebra} $\H(\infty)$ is naturally embedded into
$\A(\GLB)$, moreover, it is a subset of one of the above ideals. This somehow explains the
appearance of the Young graph. Indeed, $\Y$ (without any labels, since $\H(\infty)$ contains the
unit element) is the Bratteli diagram of $\H(\infty)$ as follows from the fact that $\H(\infty)$
and the group algebra of $S(\infty)$ are isomorphic, see e.g.\ \cite{VK_Hecke}.

\bigskip

We further concentrate on the representation theory of $\A(\GLB)$ and $\GLB$. It is well-known that
``big'' groups, such as $S(\infty)$, $U(\infty)$, $\GL{\infty}$ are \emph{wild} and in order to get
a well-behaved representation theory one has to restrict the class of the representations under
consideration. There are two approaches here, one of them deals with von Neumann factor
representations, see e.g.\ \cite{Th_characters}, and another one studies representations of
$(G,K)$--pairs, see e.g.\ \cite{Olsh_Howe}. In both approaches representations are in
correspondence with characters (or traces) of a group (or algebra), thus, it is crucial to obtain
the classification of traces.

In Section \ref{Section_character_classif} we describe the set of finite traces of $\A(\GLB)$. We
show (see Theorem \ref{theorem_characters_of_GLB}) that this set is a simplicial cone with extreme
rays parameterized by triplets $(f,\alpha,\beta)$, where $\alpha=\{\alpha_i\}$ and
$\beta=\{\beta_i\}$ are non-increasing sequences of non-negative reals satisfying:
$$
 \alpha_1\ge\alpha_2\ge\dots\ge 0,\quad \beta_1\ge\beta_2\ge\dots\ge 0,\quad \sum_{i=1}^{\infty}
 (\alpha_i+\beta_i)\le 1,
$$
and $f$ is an element of a certain (explicit) countable set $\CY'$. This result is to a large
extent based on the classification theorem for characters of the infinite symmetric group
$S(\infty)$ first proved by Thoma in \cite{Thoma1}. For each extreme trace parameterized by
$(f,\alpha,\beta)$ we give a formula for its values on $\A(\GLB)$ in terms of the values of
characters of irreducible representations of $\GL{n}$.

\subsection{Unipotent traces}

Among the traces of $\A(\GLB)$ we distinguish a class of \emph{unipotent} traces, which are
closely related to the same-named representations of $\GL{n}$. In our classification of traces the
unipotent extreme ones are such that $f=f_0$, where $f_0$ is a certain special element of $\CY'$.
They are distinguished by the fact that the values on (at least some) elements of
$\H(\infty)\subset \A(\GLB)$ are non-zero. On the contrary, if $f\ne f_0$, then for arbitrary
$\alpha$ and $\beta$ the corresponding trace vanishes on $\H(\infty)$. Moreover, the restriction
of extreme unipotent trace on $\A(\GLB)_n\simeq \mathbb C(\GL{n})$ is a linear combination of
traces of irreducible unipotent representations of $\GL{n}$, see \cite{Steinberg},
\cite{James_unipotent} for more information on unipotent representation of $\GL{n}$.

The extreme unipotent traces are in one-to-one correspondence with \emph{extreme traces of
Iwahori--Hecke algebra} $\H(\infty)$ (and, thus, with characters of the infinite symmetric group
$S(\infty)$). Extreme  traces of $\H(\infty)$ are parameterized (see \cite{VK_Hecke} and also
\cite[Section 7]{Me}) by two sequences
$$
 \alpha_1\ge\alpha_2\ge\dots \ge, \quad \beta_1\ge\beta_2\ge\dots \ge 0,\quad \sum_i
 (\alpha_i+\beta_i)\le 1
$$
and the (normalized) restriction of the extreme unipotent trace of $\A(\GLB)$ on
$\H(\infty)\subset\A(\GLB)$ is extreme  trace of $\H(\infty)$ with the same parameters. This
provides an infinite-dimensional analog of the well-known correspondence between irreducible
representations of $\H(n)$ and unipotent irreducible representations of $\GL{n}$, see e.g.\
\cite{CF}.

\bigskip

The extreme unipotent traces have a number of intriguing properties which we discuss in Section
\ref{Section_unipotent_values}.

Recall the following \emph{multiplicativity property} for the characters of $S(\infty)$, see
\cite{Thoma1}, \cite{VK_S}, \cite{KOO}. Let $\chi$ be a normalized character (i.e.\ central
positive-definite function satisfying $\chi(e)=1$) of $S(\infty)$. Then $\chi$ is extreme (i.e.\
extreme point of the convex set of all normalized characters) if and only if the following
multiplicativity property is satisfied. For $g\in S(n)\subset S(\infty)$ and $h\in S(m)\subset
S(\infty)$ let $g\odot h$ denote the element of $S(n+m)\subset S(\infty)$ obtained by adjoining the
permutations $g$ and $h$, i.e.\ by making $g$ act on $1,\dots,n$ and $h$ act on $n+1,\dots,n+m$.
Then for all $n,m,g,h$ we have
\begin{equation}
\label{Thoma_multiplicativity}
 \chi(g\odot h)=\chi(g)\chi(h).
\end{equation}
We prove (see Theorem \ref{Theorem_multiplicativity_of_characters}) that extreme unipotent traces
of $\A(\GLB)$ satisfy an analogue of \eqref{Thoma_multiplicativity}. More precisely, let
$e_g$ denote the element of $\A(\GLB)_n\subset\A(\GLB)$ corresponding to $g$ under the
identification of $\A(\GLB)_n$ and group algebra of $\GL{n}$. Further, for $g\in\GL{n}$,
$h\in\GL{m}$ let $g\odot h$ denote the block-diagonal matrix in $\GL{n+m}$ with blocks $n$ and
$m$. Then every \emph{normalized} extreme unipotent trace $\chi$ of $\A(\GLB)$ satisfies
$$
 \chi(e_{g\odot h})=\chi(e_g)\chi(e_h).
$$
for any $g$, $h$ with coprime characteristic polynomials.
We further compute the values of extreme unipotent traces on arbitrary elements of $\A(\GLB)$ in
terms of explicit \emph{specializations} (i.e.\ homomorphisms into $\mathbb C$) of the algebra of
symmetric functions $\Lambda$, see Theorem \ref{theorem_character_at_simple_class}.

\bigskip

We continue the discussion of unipotent traces by showing their relation to certain
\emph{probability measures.} We prove (see Theorem \ref{theorem_character_GLB_is_measure}) that
each extreme unipotent trace of $\A(\GLB)$ gives rise to a probability measure on the subgroup
$\B\subset\GLB$ of upper-triangular matrices and, moreover, if two traces are not proportional,
then corresponding measures are distinct. Also each such measure on $\B$ can be naturally extended
to a signed (i.e.\ not necessary positive) $\sigma$-finite measure on $\GLB$ invariant under
conjugations (by elements of $\GLB$). This gives an interpretation of traces of $\A(\GLB)$ as
characters of the group $\GLB$ which are infinite (i.e.\ not well-defined) on the group itself, but
have a \emph{ singularity of type measure.}

\medskip

We further study the properties of measures on $\B$ corresponding to the unipotent traces. We show
(see Theorem \ref{theorem_character_GLB_is_measure}) that each such measure is ergodic (with
respect to conjugations). Motivated by this connection we turn our attention to study of ergodic
probability measures on $\B$.

At this point in order to simplify the exposition it is convenient to switch from $\GLB$ to
another distinguished infinite-dimensional matrix group $\GLU$. $\GLU$ is the group group of all
almost uni-uppertriangular infinite matrices, i.e.\
$$
 \GLU=\{[X_{ij}]\in \GLB : X_{ii}=1\text{ for large enough }i\}.
$$
We remark that matrices in $\GLU$ have well-defined determinants. Thus, we can also define $\SLU\subset\GLU$ as
a subgroup of the matrices with determinant $1$; $\SLU$ can be viewed as $n=\infty$ analogue of groups $\SL{n}$.

The whole theory for $\GLU$ is very much parallel to that of $\GLB$, and we summarize it in the
Appendix. For now we only need the fact that $\GLU$ also has a distinguished class of extreme
unipotent traces, enumerated by the very same sequence $\alpha$, $\beta$, moreover, under the
identification $\A(\GLB)_n\simeq \mathbb C(\GL{n}) \simeq \A(\GLU)_n$ unipotent traces of $\GLB$
and $\GLU$ are \emph{the same} functions.

When we switch from $\GLB$ to $\GLU$ Borel subgroup $\B$ gets replaced by $\U\subset\GLU$ which is
the group of unipotent upper-triangular matrices. Note that, generally speaking, conjugations by
elements of $\GLU$ do not preserve $\U$. However, we can still define a conjugation-invariant
measure $\mu$ through the property $\mu(X)=\mu(Y)$ for every measurable $X\subset U$ and $Y\subset
U$ such that $X=gYg^{-1}$ for some $g\in\GLU$. We state and
 prove a partial result towards the \emph{classification theorem} for ergodic conjugation-invariant measures on $\U$,
 see Conjecture \ref{Conj_Kerov} and Proposition \ref{prop_list_of_ergodic}. This theorem is a particular case of
a general statement describing specializations of the algebra of symmetric functions $\Lambda$
non-negative on Macdonald polynomials (our case corresponds to Hall-Littlewood polynomials), which
is known as Kerov's conjecture, see \cite[Section II.9]{Kerov_book}. We also state a
\emph{conjectural Law of Large Numbers} for ergodic measures, see Conjecture \ref{Conj_Kerov}. One
particular case of this conjecture was proved by Borodin \cite{Borodin_GL_short}, \cite{Borodin_GL}
who studied uniform measure on $\U$. Finally, in Theorem \ref{Theorem_prob_measure_for_character}
 we explain which measures in the above conjectural classification
correspond to extreme unipotent traces.

\subsection{Construction of representations of $\GLB$}

Our next topic is the construction of representations of $\GLB$ corresponding to the extreme
unipotent traces of $\A(\GLB)$. Of course, there exists an abstract general
(Gelfand-Naimark-Segal) construction for the representation with given trace, but since by its
definition $\GLB$ is a transformation group we seek for  more explicit constructions based on its
\emph{natural} action in the infinite-dimensional vector space over $\F$.

In Section \ref{Section_reps} we adopt the representation formalism of \cite{Olsh_Howe} in a
modified form and construct unitary representations $T$ of $(\GLB\times\GLB)$ (in other words, we
consider two-sided representations) in a Hilbert space $H$ possessing a distinguished vector $v$
(in general, $v$ might belong to a certain extension of $H$), which is cyclic and invariant under
the action of $\GLB$ diagonally embedded into $(\GLB\times\GLB)$. $v$ defines a spherical function
$\chi(a)=((a,e)v,v)$, and viewed as a function on $\A(\GLB)$ (or $\GLB$) this function becomes our
trace. There is a simple link between our constructions and factor representations (semifinite, in
general). If we consider the restriction of $T$ on the first component of $\GLB\times\GLB$, then we
get a \emph{von Neumann factor representation.}

In the most well-studied settings of the asymptotic representation theory, e.g.\ for infinite
symmetric group $S(\infty)$ (see \cite{Thoma1}, \cite{VK_S}, \cite{Olsh_S}, \cite{Ok_S}) and real
infinite-dimensional matrix groups such as $U(\infty)$ (see \cite{Vo}, \cite{Olsh_Howe_short},
\cite{Olsh_Howe}) distinguished vector $v$ belongs to $H$ and corresponding factor representation
is of type $II_1$. However, in our setting, since our traces are defined only on $\A(\GLB)$ instead
of the whole group $\GLB$, we have to use generalized vectors (distributions) $v$ and corresponding
factor representations we get are of type $II_\infty$.

Similar situations appeared before in the investigation of at least two topics of classical
representation theory of finite-dimensional groups. In the study of unitary representations of
general Lie groups sometimes one is led to consider type $II_\infty$ factor representation, see
e.g.\ \cite{Pu1}, \cite{Pu2} where characters and representations of simply connected solvable and
more general Lie groups are studied. From the other side, in the theory of semisimple Lie groups
generalized distinguished vectors show up\footnote{Note, however, that all the representations are
of type $I$ in this theory.} in harmonic analysis, i.e.\ when one tries to decompose
highly-reducible representations. As a quick example, by the well-known theorem bi-regular
representation of a finite group $G$ equipped with a distinguished vector $v$ ---
$\delta$-function at identity element of $G$ --- is the direct sum of irreducible spherical
representations $\pi^\lambda\otimes (\pi^\lambda)^*$ of Gelfand pair $(G\times G,G)$ (here
$\lambda$ goes over all irreducible representations of $G$). By Peter-Weyl theorem, the same
decomposition is valid for a compact Lie group, however, $\delta$--function at identity is no
longer a vector of $L_2$ on the group, rather it is a generalized vector (distribution).

Coming back to the construction of representation of $\GLB$ recall that extreme unipotent traces
of $\A(\GLB)$ are parameterized by two sequences of non-negative reals $\{\alpha_i\}$ and
$\{\beta_i\}$. We start from two distinguished simplest cases, where the corresponding factor
representations are of type $I$. If $\alpha_1=1$ with other parameters being zeros, then the
corresponding representations is trivial $1$--dimensional representation of $\GLB$. On the
contrary,  if $\beta_1=1$ with other parameters being zeros, then the corresponding von Neumann
factor representation is of type $I_\infty$ and the construction is related to classical
\emph{Steinberg representation} of $\GL{n}$, see \cite{Steinberg}, \cite{Hum}.

Further we concentrate on the case $\beta_i=0$ for all $i$ and $\sum_i \alpha_i=1$, where we get
type $II_\infty$ factor representations. To simplify the exposition let us stick to the case when
only $\alpha_1$ and $\alpha_2$ are non-zero. Our construction starts with natural action of $\GLB$
on the \emph{grassmanian} $Gr(V)$, which is the set of all subspaces of countable
infinite-dimensional vector space $V$ over $\F$. (If more than two $\alpha_i$s are non-zero, then
$Gr(V)$ gets replaced by an appropriate space of flags.) A well-known group-measure space (or
crossed product) construction going back to the papers of F.~J.~Murray and J.~von~Neumann
\cite{MN}, \cite{N} associates a von Neumann factor to a free ergodic action of a group on a space
equipped with measure. Of course, the action of $\GLB$ on $Gr(V)$ is not free and, thus,
modifications are necessary. The known solution here is to use \emph{principal grouppoid} of the
equivalence relation spanned by the group action, in other words, we take the set $Gr^2(V)\subset
Gr(V)\times Gr(V)$ which is the graph of the equivalence relation. Grouppoids and construction of
the associated von Neumann algebras attracted a lot of attention in the literature, see
\cite{Moore} for a review, a somewhat simpler case when all the classes of equivalence relation are
finite was first studied in \cite{Kr} and \cite{FM}. Grouppoids in the context of asymptotic
representation theory of symmetric groups first appeared in \cite{VK_S}, where the construction was
based on the action of $S(\infty)$ on the set of words equipped with product measure.

When classes (orbits) are uncountable (which is the case for the action of $\GLB$ in $Gr(V)$) the
construction is more delicate, and even the  definition of the correct measure on grouppoid
becomes complicated. General solutions do exist here, see \cite{Hahn}, \cite{Renault}, however, in
our case the situation is simplified by the fact that $Gr(V)$ can be decomposed into
\emph{Schubert cells}, each of which is a $\B$--orbit. This lets us to start from a measure on
symbols of Schubert cells (here we use the Bernoulli measure, similarly to the constructions of
\cite{VK_S} for $S(\infty)$) and produce the measures on $Gr(V)$ and grouppoid $Gr^2(V)$ using it.
We end up with quasiinvariant measures, which allow us to construct usual unitary representation
of $\GLB\times\GLB$ in the space of square-integrable functions on $Gr^2(V)$. As for the
distinguished vector $v$, in the case of countable equivalence classes (as happens for
$S(\infty)$) the right choice is known to be the indicator function of the diagonal of grouppoid,
see \cite{VK_S}, \cite{FM}. In our case the diagonal has measure zero, so this choice is
unappropriate. Because of that we have to use a generalized vector (distribution) $v$ which is the
integral along the diagonal (defined only for the continuous functions.)

One interesting aspect here is that the measure on $Gr(V)$ we use is not $\GLB$--invariant and,
moreover, one proves that there is no equivalent $\sigma$--finite $\GLB$--invariant measure. In
classics this would imply that the resulting crossed product gives a factor of type $III$, see
e.g.\ \cite[Theorem 2.4]{Kr}, \cite[Theorem I.3.12]{SV}, while we end up with type $II_\infty$
factor representations of $\GLB$. An explanation here is the following: our trace on the operators
of representation of $\GLB$ can not be extended to the operators of the multiplication by the
function, as opposed to the situations related to the variations of the crossed product
construction.

Somewhat related question concerns cyclicity of the distinguished vector $v$. In general, there is
no guarantee that $\GLB\times\GLB$ orbit of our vector is dense, thus, we have to consider the
representation in the cyclic hull of $v$. In \cite{VK_S}, \cite[Section 5]{Olsh_S},
\cite{V_nonfree} the cyclicity question for the representations of infinite symmetric group
$S(\infty)$ was discussed; for $\GLB$ this topic requires further investigations.

\bigskip

We are currently unable to construct representations of $\GLB$ corresponding to other unipotent
traces of $\A(\GLB)$. Even the case $\alpha_i=\beta_i=0$ for all $i$, which for the infinite
symmetric group $S(\infty)$ corresponds to the biregular representation, is out of our reach at
the moment.

In papers \cite{V_nonfree}, \cite{V_nonfree2} a new approach to the construction of spherical
representations (or finite factor representations on the different language) related to infinite
symmetric group $S(\infty)$ was proposed. The authors hope that this approach might be extended to
$\GLB$.
\subsection{Biregular representation of $\GLB$.}
The final object of our interest is the (bi-)regular representation of $\GLB$. Since $\GLB$ admits
a unique Haar measure $\mu_{\GLB}$ (normalized by the condition $\mu_{\GLB}(\B)=1$) there is a
well-defined two-sided representation of $\GLB$ in $L_2(\GLB, \mu_{\GLB})$. This representation
equipped with the distinguished distribution--- $\delta$--function at unit element of the group ---
fits into the formalism of generalized spherical representations and corresponds to a certain
trace of $\A(\GLB)$. As in the classical harmonic analysis on finite or compact groups we are
interested into the decomposition of this representation. In Section \ref{Section_harmonic} we
describe the decomposition of the trace of bi-regular representation of $\GLB$ into a combination of
extreme traces of $\A(\GLB)$.

\label{section_intro_before_comments}

\subsection{Motivations and comments}
\label{Section_comments}

A well-known point of view is that the symmetric group $S(n)$ can be viewed as $\GL{n}$ over the
field with one element, i.e.\ with $q=1$. This agrees with similarities between the representation
theory of $\GL{n}$ and $S(n)$, so it is natural to expect some similarities for $n=\infty$ as
well.

The infinite symmetric group $S(\infty)$ is usually defined as the inductive limit of finite
symmetric groups, equivalently, $S(\infty)$ consists of all bijections of countable set, which
permute only finitely many elements. A natural adaptation of this definition to $\GL{n}$ is the
following. Realize $\GL{n}$ as a subgroup of $\GL{n+1}$ acting in the space spanned by the first
$n$ coordinate vectors and fixing $n+1$st coordinate vector and consider the inductive limit of
$\GL{n}$ with respect to such embeddings. In this way we get the infinite-dimensional group
$\GL{\infty}$. However, the representation theory of $\GL{\infty}$ turns out to be not as rich as
one could hope for. For instance, the set of extreme (indecomposable) characters of $\GL{\infty}$
is countable (see \cite{ThomaGL}, \cite{Skudlarek}) as opposed to the infinite symmetric group
$S(\infty)$ (see \cite{Thoma1}, \cite{VK_S}, \cite{KOO}, \cite{Ok_S}) or infinite-dimensional
unitary group $U(\infty)$ (see \cite{Vo}, \cite{VK_U}, \cite{Bo}, \cite{OkOlsh}, \cite{BO},
\cite{P}, \cite{GP}) for which such sets comprise infinite-dimensional domains in $\mathbb
R^{\infty}$. This leads one to seek for other $n=\infty$ analogue of $\GL{n}$.

The key idea here is to change the embeddings. A new definition is hinted by the
notions of \emph{parabolic induction and restriction} well-known in the representation theory of
$\GL{n}$, see \cite{Green}, \cite{Zel}, \cite{F}. This leads to \emph{parabolic embeddings}
$i_n:\mathbb C(\GL{n}) \hookrightarrow \mathbb C(\GL{n+1})$ which are no longer induced by the
group embeddings, see Section \ref{section_A_as_inductive_limit} for the formal definition. The
inductive limit of $\mathbb C(\GL{n})$ with respect to the embeddings $i_n$ is our main hero ---
algebra $\A(\GLB)$. A thorough analysis of the definitions leads to the realization of $\A(\GLB)$
as subalgebra of the algebra of the functions on a group, that's how the group $\GLB$ first
appears. Note that $\GL{\infty}$ is a dense subgroup of $\GLB$, so another point of view might be
to consider $\GLB$ as a certain \emph{completion} of discrete group $\GL{\infty}$.

The representation theory of $\GLB$, indeed, turns out to be similar to that of $S(\infty)$.
First, the classification of traces of $\A(\GLB)$ and characters of $S(\infty)$ are similar,
sequences $\{\alpha_i\}$ and $\{\beta_i\}$ appear in both. The similarity is even more striking
when one considers unipotent extreme traces of $\A(\GLB)$. Their normalized versions are in
one-to-one correspondence with extreme characters of $S(\infty)$ and both families have similar
properties, e.g.\ multiplicativity and the same coefficients of decomposition into irreducibles of
restrictions to $\A(\GLB)_n$ ($S(n)$).

The constructions of the representations corresponding to characters are also similar, although
some distinctions do exist (e.g.\ the distinguished vector becomes a distribution for $\GLB$).
More precisely, the realization of representations of $\GLB$ with unipotent traces with non-zero
parameters $\alpha_i$ is related to the spaces of flags of subspaces, while corresponding
representations of $S(\infty)$ are related to its exact $q=1$ analogue which is the space of flags
of subsets, see \cite{VK_S}, \cite{VT}.

The above facts let us claim that the group $\GLB$ might be the right $q$-analogue of $S(\infty)$
and $n=\infty$ analogue of $\GL{n}$ in the context of the asymptotic representation theory.

However, some of the similarities break down when we start considering representations with
non-zero $\beta_i$. Representation of $S(\infty)$ with single non-zero parameter $\beta_1=1$ is
the simple one-dimensional alternating representation, while the corresponding representation of
$\GLB$  is an infinite-dimensional one; this is parallel to the difference between alternating
one-dimensional representation of $S(n)$ and corresponding unipotent (principal) representation of
$\GL{n}$ which is the Steinberg representation of dimension $q^{n(n-1)/2}$.

More importantly, while the (bi-)regular representation of $S(\infty)$ is irreducible and
corresponds to zero parameters $\alpha_i$ and $\beta_i$, the (bi-)regular representation of $\GLB$
is reducible (as we explain in Section \ref{Section_harmonic}). The construction of the unipotent
representation of $\GLB$ corresponding to zero parameters at the moment remains unknown.

We intensively exploit the similarity between $S(\infty)$ and $\GLB$ in our methods. For instance,
some theorems of the present article are based on the Ring Theorem, which originally was
discovered in the study of $S(\infty)$, see \cite{VK_ring}, \cite{Kerov_book} and also
\cite[Section 8.7]{GO_Ring}). Also Schur polynomials play an important role in the study of
$S(\infty)$, while in the present paper we intensively use both Schur polynomials and their
$q$--deformation --- Hall-Littlewood polynomials.

In the classics, the representation theory of $S(\infty)$ has numerous connections with the
representation theory of $U(\infty)$, see \cite{BO2} and references therein. A $q$--deformation of
the character theory of $U(\infty)$ related to the quantum groups was proposed in \cite{G}. It is
yet to discover whether the representation theory of $\GLB$ is somehow related to that
$q$--deformation.

Finally, we remark that some results on the structure of $\GLB$ from the algebraic point of view
can be also found in the literature, see \cite{Ho}, \cite{GuHo} and references therein.

\subsection{List of main notations and theorems}
\begin{flushleft}
{\bf Notations:}

 $S(n)$ --- symmetric group of rank $n$

 $\H(n)$ --- Iwahori--Hecke algebra of rank $n$

 $\F$ --- finite field with $q$ elements

 $\GL{n}$ --- group of all invertible $n\times n$ matrices over $\F$

 $S(\infty)$, $\H(\infty)$, $\GL{\infty}$ --- inductive limits of corresponding finite $n$ objects

 $\GLB$ --- group of all almost uppertriangular infinite matrices over $\F$

 $\B$ --- group of all uppertriangular infinite matrices over $\F$

 $\BI_n$ --- subgroup of $\B$ of all matrices such that their top left $n\times n$ corner is the
 identity matrix

 $\B_n$ --- group of all uppertriangular $n\times n$ matrices over $\F$

 $\GLU$ --- group of all almost uni-uppertriangular infinite matrices over $\F$

 $\U$ --- group of all uni-uppertriangular infinite matrices over $\F$

 $\mu_\GLB$ --- Haar measure on $\GLB$ normalized by $\mu_\GLB(\B)=1$

 $\mu_\GLU$ --- Haar measure on $\GLU$ normalized by $\mu_\GLU(\U)=1$

 $\A(\GLB)$, $\A(\GLU)$ --- algebra of all continuous functions  with compact support (on the corresponding group) taking only
finitely many values

\bigskip

 $\Y$ --- set of all Young diagrams and also Young graph

 $n(\lambda)$ --- function of Young diagram $\lambda=\lambda_1\ge\lambda_2\ge\dots$ given by $\sum_i (i-1)\lambda_i$

 $\C$ --- set of all irreducible monic polynomials over $\F$ other than $x$ and $1$

 $\C_n$ --- all degree $n$ polynomials in $\C$

 $\CY_n$ --- set of all maps from from $\C$ to $\Y$ of degree $n$

 $\CY$ --- disjoint union of sets $\CY_n$, $n=1,2,\dots$

 $\pi^f$, $\chi^f$ --- irreducible complex representation of $\GL{n}$ parameterized by $f\in\CY_n$ and
 its conventional character

 $\CY'$ --- subset of $\CY$ of maps $f$ such that $f(``x-1'')=\emptyset$.

\bigskip

 $\Lambda$ --- algebra of symmetric (polynomial) functions in countably many variables

 $h_n$, $e_n$, $p_n$ --- complete homogeneous functions, elementary symmetric functions and Newton power
 sums, respectively

 $s_\lambda$ --- Schur function indexed by $\lambda\in\Y$

 $P_\lambda(\cdot; t)$, $Q_\lambda(\cdot; t)$ --- Hall-Littlewood $P$ and $Q$ functions with
 parameter $t$, indexed by $\lambda\in\Y$

 $\Sp_{\alpha,\beta,\gamma}$ --- homomorphism from $\Lambda$ into $\mathbb C$ indexed by two
 sequence of non-negative numbers $\alpha=\{\alpha_i\}$, $\beta=\{\beta_i\}$ and real number $\gamma$ such that
 $\sum_i(\alpha_i+\beta_i)\le \gamma$, and given by its values on power sums
 $$
  \Sp_{\alpha,\beta,\gamma}[p_1]=\gamma,\quad  \Sp_{\alpha,\beta,\gamma}[p_k]=\sum_i \alpha_i^k +
  (-1)^{k-1}\sum_i \beta_i^k,\, k>1
 $$

\bigskip
{\bf Key theorems: }
\end{flushleft}

\emph{Proposition \ref{proposition_A_as_inductive_limit}} on page
\pageref{proposition_A_as_inductive_limit} identifies $\A(\GLB)$ with the inductive limit of the
group algebras $\mathbb C(\GL{n})$.

\emph{Theorem \ref{theorem_characters_of_GLB}} on page \pageref{theorem_characters_of_GLB}
provides the description of all extreme traces of $\A(\GLB)$.

\emph{Theorem \ref{Theorem_multiplicativity_of_characters}} on page
\pageref{Theorem_multiplicativity_of_characters} gives the proof of multiplicativity of extreme
unipotent traces of $\A(\GLB)$.

\emph{Theorems \ref{theorem_character_at_unipotent_class} and
\ref{theorem_character_at_simple_class}} on page \pageref{theorem_character_at_unipotent_class}
 relate the values of extreme unipotent characters
to specializations of Hall--Littlewood polynomials.

\emph{Theorem \ref{Theorem_restriction_to_Hecke}} on page \pageref{Theorem_restriction_to_Hecke}
identifies the restrictions of unipotent traces with extreme traces of Iwahori--Hecke algebra.

\emph{Theorems \ref{theorem_character_GLB_is_measure} and
\ref{Theorem_prob_measure_for_character}} on pages \pageref{theorem_character_GLB_is_measure} and
\pageref{Theorem_prob_measure_for_character} explain that each unipotent characters can be viewed
as a probability measure.

\emph{Conjecture \ref{Conj_Kerov}} on page \pageref{Conj_Kerov} gives the (conjectural)
classification and law of large numbers for conjugation--invariant probability measures on
infinite upper-triangular matrices.

\emph{Theorem \ref{Theorem_construction_gras}} on page \pageref{Theorem_construction_gras}
provides a construction for the representations of $\GLB$ related to grassmanian.

\emph{Theorem \ref{theorem_construction_flags}} on page \pageref{theorem_construction_flags}
provides a construction for the representations of $\GLB$ related to spaces of flags.

\emph{Theorem \ref{Theorem_biregular}} on page \pageref{Theorem_biregular} describes the
decomposition of the biregular representation of $\GLB$.

\label{Section_intro_notations}

\subsection{Acknowledgements}
 During the long period of  first stage of the
work on the subject the last two authors (S.K.\, and A.V.) had many useful discussions with
A.~Zelevinsky, G.~Olshanski, J.~Bernstein. The first author (V.G) is grateful to G.~Olshanski for
drawing his attention to $\GLB$ and fruitful discussions at various stages of this work. V.G.\, was
partially supported by RFBR-CNRS grants 10-01-93114 and 11-01-93105. A.V.\ was partially supported
by RFBR grants  11-01-12092 (OFI-m), 11-01-00677 and 13-01-12422 (OFI-m).

\section{The group $\GLB$ and its Schwartz--Bruhat algebra $\A(\GLB)$}
\label{Section_AG_as_semisimple}

\subsection{Basic definitions}

Let $\F$ be the finite field with $q$ elements and let $\GL{n}$ denote the group of all invertible
$n\times n$ matrices over $\F$. For any matrix $X$ we denote through $X^{(n)}$ its top left
$n\times n$ corner.

\begin{definition} $\GLB$ is the group of all invertible almost upper-triangular matrices over $\F$
in other words $X=[X_{ij}]_{i,j=1}^\infty$ is an element of $\GLB$ if there exists $n$ such that:
\begin{enumerate}
\item The $n\times n$ submatrix $X^{(n)}$ is invertible,
\item $X_{ij}=0$ for all $i$ such that $i>j$ and $i>n$,
\item $X_{ii}\ne 0$ for $i>n$.
\end{enumerate}
\end{definition}

The group $\GLB$ is an inductive limit of groups $\GLB_n$, where
$$
 \GLB_n=\{ [X_{ij}]\in\GLB\mid X_{ij}=0\text{ if both } i>j\text{ and }i>n\},
$$
in particular, $\GLB_0=\B\subset \GLB$ is the subgroup of all upper-triangular invertible
matrices.

Each $\GLB_n$ is a compact group (with topology of pointwise convergence of matrix elements).
$\GLB$ as an inductive limit of $\GLB_n$ is a locally compact topological group. Let $\mu_{\GLB}$
denote the biinvariant Haar measure on $\GLB$ normalized by the condition $\mu_{\GLB}(\B)=1$.

The space $L_1(\GLB,\mu_{\GLB})$ is a Banach involutive algebra with multiplication given by the
convolution.

\begin{definition}  $\A(\GLB)$  is defined as the subalgebra of $L_1(\GLB,\mu_{\GLB})$ formed by all locally
constant functions with compact support. In other words, a function $f(X)$ belongs to $\A(\GLB)$
if their exists $n$ and a function $f_n: \GL{n}\to\mathbb C $ such that:
$$
f(X)=\begin{cases} f_n(X^{(n)}),\text{ if } X\in\GLB_n, \\0,\text{ otherwise.}
\end{cases}
$$
\end{definition}

Clearly, $\A(\GLB)$ is dense in $L_1(\GLB,\mu_{\GLB})$. Note that algebra $\A(\GLB)$ does not have
a unit element.

\begin{definition} A (linear) function $\chi: \A(\GLB) \to \mathbb C$ is a \emph{trace} of $\A(\GLB)$ if
\begin{enumerate}
 \item $\chi$ is central, i.e.\ $\chi(WU)=\chi(UW)$,
 \item $\chi$ is positive definite, i.e.\ $\chi(W^*W)\ge 0$ for any $W\in \A(\GLB)$,
\end{enumerate}
\end{definition}
\noindent {\bf Remark. } It is impossible to normalize the  traces, i.e.\ for any $a\in\A(\GLB)$
there exists a  trace $\chi$ such that $\chi(a)=0$.

\smallskip

A  trace $\chi$ is \emph{indecomposable} if $\chi=\alpha_1\chi_1+\alpha_2\chi_2$ with
$\alpha_1>0$, $\alpha_2>0$ implies that both $\chi_1$ and $\chi_2$ are multiples of $\chi$. In
other words, indecomposable  traces are elements of extreme rays of the convex cone of all
 traces.

\subsection{$A(\GLB)$ as an inductive limit}

\label{section_A_as_inductive_limit}

For any matrix $g\in\GL{n}$ let $I^\GLB_g\in\A(\GLB)$  denote the indicator function
$$
 I^\GLB_g(X)=
 \begin{cases}
  1, \text{ if } X\in\GLB_n \text{ and }X^{(n)}=g, \\
  0, \text{ otherwise.}
 \end{cases}
$$
Let $e(n)$ denote the identity element of $\GL{n}$. Then by the definition
$$I^{\GLB}_g(X)=I^{\GLB}_{e(n)}(Xg^{-1})=g\cdot I^{\GLB}_{e(n)}.$$

\begin{definition} $\A(\GLB)_n$ is defined as the linear span of $I^\GLB_g$, $g\in\GL{n}$. Put it otherwise,
$\A(\GLB)_n$ consists of functions from $\A(\GLB)$ with support in $\GLB_{n}$ and depending only
on the restriction of operator $g\in\GLB_n$ on $V_n\subset V_\infty$.
\end{definition}

Let $\mathbb
 C(\GL{n})$ denote the conventional group algebra of $\GL{n}$, i.e.\ the algebra with linear basis
 $\{e_g\}_{g\in\GL{n}}$ and multiplication given by $e_g e_h=e_{gh}$.
   The following proposition is straightforward
\begin{proposition}
\label{prop_A_n_as_group_algebra}
 $\A(\GLB)_n$ is a subalgebra of $\A(\GLB)$ isomorphic to the group algebra $\mathbb
 C(\GL{n})$.  The isomorphism is given by $e_g \to (q-1)^n q^{n(n-1)/2} I^\GLB_g$
\end{proposition}

Observe that $\A(\GLB)_n\subset \A(\GLB)_{n+1}$. In the basis $I^\GLB_g$ this inclusion is given
by
$$
 i_n: I^\GLB_g \to \sum_{h\in Ext^\GLB(g)} I^\GLB_h,
$$
where for $g\in\GL{n}$ we have
\begin{multline*}
 Ext^\GLB(g)=\biggl\{[h_{ij}]\in\GL{n+1}\mid\\ h^{(n)}=g\text{ and }
 h_{n+1,1}=h_{n+1,2}=\dots={h_{n+1,n}}=0\biggr\}.
\end{multline*}

Summarizing the discussion of this section we get the following statement.

\begin{proposition}
\label{proposition_A_as_inductive_limit} The algebra $\A(\GLB)$ can be identified with the
inductive limit of algebras $\A(\GLB)_n$:
$$
 \A(\GLB)=\ilim_{n\to\infty} \A(\GLB)_n = \bigcup_{n} \A(\GLB)_n.
$$
For every $n$ the algebra $\A(\GLB)_n$ is isomorphic to the group algebra $\mathbb
 C(\GL{n})$
\end{proposition}

Thus, $\A(\GLB)$ is a \emph{locally semisimple} algebra.

\subsection{Facts from representation theory of $\GL{n}$}
Let us fix the notations and recall some basic facts from the representation theory of the group
$\GL{n}$ which immediately translate into the statements for the representations and  traces of
algebra $\A(\GLB)_n$. To a large extent we adopt the notations of the book \cite{M}.

A Young diagram $\lambda$ is a finite collection of boxes arranged in rows with nonincreasing row
lengths $\lambda_i$. The total number of boxes in $\lambda$ is denoted by $|\lambda|$. Let $\Y$
denote the set of all Young diagrams. We agree that the empty set $\emptyset\in \Y$ and
$|\emptyset|=0$. $\Y_n\subset\Y$ stays for the set of all Young diagrams with $n$ boxes. Also for
the Young diagram $\lambda$ its transpose diagram is denoted $\lambda'$; the row lengths of
$\lambda$ coincide with column lengths of $\lambda'$. For a box $\square\in\lambda$ its hook length
$h(\square)$ is one plus number of the boxes below $\square$ (in the same column) plus number of
the boxes to the right from $\square$ (in the same row). Finally, we set $n(\lambda)=\sum_i
(i-1)\lambda_i =\sum_{i} {\lambda'_i\choose 2}.$

For $d>1$ let $\C_d$ denote the set of all monic irreducible polynomials of degree $d$ over $\F$.
Let $\C_1$ be the set of all linear polynomials $x-a$, $a\in\F^*$, i.e.\ we exclude the polynomial
$x$. Clearly, $|\C_1|=q-1$. Let $\C=\bigcup_{d=1}^{\infty} \C_d$.
\begin{definition}
 A \emph{family of Young
diagrams} over the set $\C$ is a map
$$
\phi: \C\to\Y,
$$
such that
$$
 |\phi|:=\sum_d\sum_{c\in\C_d} d|\phi(c)|<\infty.
$$
We call $|\phi|$ the \emph{degree} of $\phi$.
\end{definition}

Let $\CY_k$ denote the set of all families of degree $k$ and define $\CY=\bigcup_{k=0}^\infty
\CY_k$.

\begin{theorem}[Green]
\label{theorem_irreps_of_gl}
 Irreducible representations of $\GL{k}$ are parameterized by elements of $\CY_k$.
 The dimension of the irreducible representation parameterized by $\phi\in\CY_k$ is given by the
 $q$--analogue of the hook formula
 $$
   {\rm dim}_q(\phi)= (q^k-1)\dots(q-1)\prod_{d\ge 1} \prod_{c\in\C_d} \frac{q^{d n(\phi(c))}} {\prod_{\square\in \phi(c)} (q^{d
   h(b)}-1)}.
 $$
\end{theorem}

For the proof, construction of the representations and their characters see \cite{Green},
\cite{Zel}, \cite{M}. For $f\in\CY_d$ let $\pi^f$ denote the corresponding irreducible
representation, $H(\pi^f)$ the space of this representation, and let $\chi^f(\cdot)$ be its
conventional character (i.e.\ matrix trace of $\pi^f(\cdot)$).

\begin{corollary}
\label{cor_decomposition_of_restriction}
 The set of all  traces of $\A(\GLB)_n$ is a simplicial cone spanned by  traces
 $\chi^f$. In other words, if $\chi^n$ is a trace of $\A(\GLB)_n$, then there exist unique
 nonnegative coefficients $c(f)$ such that
 $$
  \chi^n(\cdot)=\sum_{f\in\CY_n} c(f) \chi^f(\cdot).
 $$
\end{corollary}
\begin{proof}
 $\A(\GLB)_n$ is isomorphic to the conventional group algebra of $\GL{n}$. Under this correspondence a  trace of
 $\A(\GLB)_n$ turns into the character of $\GL{n}$, i.e.\ central (class)
 positive-definite function on the group. It is well-known that characters of a finite group form
 a cone spanned by the characters (matrix traces) of the irreducible representations.
\end{proof}

Next we describe the interrelations between  traces and inclusions $i_n$.

Embed $\GL{n-1}\times \GL{1}$ into $\GL{n}$ as the subgroup of block diagonal matrices. Consider
the subgroup $U^n_n\subset\GL{n}$ consisting of unipotent upper triangular matrices $[u_{ij}]$
such that $u_{ij}$ is non-zero only for $j=n$ (and $u_{nn}=1$). Note that $\GL{n-1}\times \GL{1}$
normalizes $U_n^n$.

\begin{theorem}
\label{Theorem_parabolic_restriction}
 Suppose that $f\in \CY_n$. Let $\widehat H(\pi^f)$ denote the subspace of $U_n^n$-invariant
 vectors in $H(\pi^f)$. And let $\widehat\pi^f$ denote the representation of $\GL{n-1}\times \GL{1}$ in this
 subspace. Let $\{f_i\}$ be all families in $\CY_{n-1}$ for which there exist $y_i\in\C_1$
 such that
 \begin{enumerate}
  \item $f_i(x)=f(x)$ for $x\ne y_i$,
  \item The difference of the Young diagrams $f(y_i)\setminus f_i(y_i)$ is a single box.
 \end{enumerate}
 Finally, let $f\setminus f_i$ denote the family from $\CY_1$ such that $(f\setminus f_i)(y_i)$ is
 the one box diagram.

 We have
 $$
  \widehat\pi^f=\bigoplus_i \pi^{f_i} \otimes \pi^{f\setminus f_i},
 $$
\end{theorem}

\begin{proof} See e.g.\ \cite[Chapter III]{Zel}.
\end{proof}

\subsection{Structure of $\A(\GLB)$}

We need to introduce some notations to state an important corollary of Theorem
\ref{Theorem_parabolic_restriction}.

\begin{definition}
 For two families $f\in\CY_n$ and $g\in\CY_{n-1}$ we say that
$g$ precedes $f$ and write $g\prec_{\GLB} f$ if the values of $f$ and $g$ on the irreducible
polynomial $``x-1"\in\C_1$ differ by one box, i.e. $f(``x-1")\setminus g(``x-1")=\square$, and for
all other polynomials the value is the same, i.e.\ $f(u)=g(u)$ for all $u\ne ``x-1"$.
\end{definition}

\begin{theorem}[Branching rule]
\label{proposition_branching_of_characters_GLB}
 Let $\pi^f$  be the irreducible representation of algebra
 $\A(\GLB)_n$ (equivalently, of the group $\GL{n}$) parameterized by $f\in\CY_n$ and let
$\chi^f$ be its conventional character (i.e.\ matrix trace). The restrictions of $\pi^f$ and
$\chi^f$ to the subalgebra
 $\A(\GLB)_{n-1}$ admit the following decomposition:
$$
\chi^f\rule[-2.5mm]{.4pt}{5mm}_{\,\A(\GLB)_{n-1}}=\sum_{g\prec_{\GLB} f} \chi^g,
$$
equivalently,
$$
 \pi^f \rule[-2.5mm]{.4pt}{5mm}_{\,\A(\GLB)_{n-1}} = \mathcal N \oplus \bigoplus_{g\prec_{\GLB} f} \chi^g,
$$
where $\mathcal N$ is a zero representations of $\A(\GLB)_{n-1}$ of dimension
$\dim(f)-\sum_{g\prec_{\GLB} f} \dim(g)$.
\end{theorem}
\noindent {\bf Remark 1. } By zero representation we mean the action of $\A(\GLB)_{n-1}$ by the
identical zero in a vector space of arbitrary dimension.

\noindent {\bf Remark 2. } Theorem \ref{proposition_branching_of_characters_GLB} implies, in
particular, that the restriction of $\pi^f$ to $\A(\GLB)_{n-1}$ is multiplicity free. This property
was mentioned by various authors, the first proof was given by A.~Zelevinsky \cite[Chapter
III]{Zel} using the Hopf algebras approach. Now there exist simple direct proofs of this fact, see
\cite{Goryachko}, \cite{AG}.

\noindent {\bf Remark 3.} As opposed to the situation with parabolic embeddings, the restrictions
of irreducible representations of $\GL{n}$ to the naturally embedded subgroup $\GL{n-1}$ are
\emph{not} multiplicity free, see  \cite{Th_restrictions}, \cite[Chapter III, Section 13]{Zel}.

\begin{proof}
 This follows from Theorem \ref{Theorem_parabolic_restriction} and we use the notations of that
 theorem. Indeed, the summation in the
 definition of parabolic embedding $i_{n-1}$ introduces averaging over $U_n^n$ and over $\GL{1}$.
 Therefore, the parabolic embedding translates into the projection on $\GL{1}$--invariants in $\widehat \pi^f$.
\end{proof}

Now the structure of locally semisimple algebra $\A(\GLB)$. can be encoded via its \emph{Bratteli
diagram} \cite{Br}, \cite{VK_AF}, \cite{Kerov_book}.
\begin{proposition}
The Bratelli diagram of algebra $\A(\GLB)$ is  a graded graph $B({\GLB})$ supplemented with
additional numbers, labels of the vertices. The set $B({\GLB})_n$ of vertices at level $n$ is
$\CY_n$. The label $l(f)$ of the vertex $f\in\B(\GLB)_n$ is the dimension of the irreducible
representation of $\GL{n}$ parameterized by $f$, the formula for its computation is given in
Theorem \ref{theorem_irreps_of_gl}. An edge joins vertex $f$ and vertex $g$ is and only if
$g\prec_{\GLB} f$.
\end{proposition}
\begin{proof} This is a reformulation of the \emph{branching rule} of
Theorem \ref{proposition_branching_of_characters_GLB}.
\end{proof}

For convenience of the reader we recast in our setting the general procedure for the reconstruction
of the involutive algebra by its Bratteli diagram.

By the well-known theorem algebra $\A(\GLB)_n$ is isomorphic to the direct sum of matrix algebras
of ranks equal to the dimensions of its irreducible representations. Therefore,
\begin{equation}
\label{eq_brat_1}
 \A(\GLB)_n=\bigoplus_{f\in B({\GLB})_n} Mat(l(f),l(f))
\end{equation}
The inclusions $i_n:\A(\GLB)_n\hookrightarrow \A(\GLB)_{n+1}$ can be reconstructed as follows. For
every $f\in B({\GLB})_n$ fix the embedding
$$
 \bigoplus_{g\prec_\GLB f} Mat(l(g),l(g)) \hookrightarrow Mat(l(f),l(f))
$$
as block-diagonal matrices. Note that here we need the inequality
$$
 \sum_{g\prec_\GLB f} l(g)\le l(f)
$$
to be satisfied. Let $i_{g,f}$ denote the above embedding considered as a map from the matrix
algebra corresponding to $g$ to the matrix algebra corresponding to $f$ viewed as a subalgebra of
$\A({\GLB})_{n+1}$.

Now for $a=\sum_{f\in B({\GLB})_n} m_f$, with $m_f\in Mat(l(f),l(f))$ in \eqref{eq_brat_1}, we set
$$
 i_n(a)=\sum_{f\in B({\GLB})_{n+1}} \sum_{g\prec f} i_{g,f} m_g.
$$
Algebra $\A(\GLB)$ is reconstructed (up to isomorphism) as the inductive limit of $\A(\GLB)_n$.

\subsection{Some subalgebras of $\A({\GLB})$}
Let $\B_n\subset\GL{n}$ be the (Borel) subgroup of all upper-triangular matrices. We call an
element
$$
 a=\sum_{g\in\GL{n}} c(g) e_g \in \mathbb C(\GL{n})
$$
$\B_n$--biinvariant if $c(g)=c(b_1 g b_2)$ for any $g\in\GL{n}$ and $b_1,b_2\in \B_n$. Put it
otherwise, $\B_n$--biinvariant element is a linear combination of characteristic functions of
double cosets $\B_n g \B_n$.

\begin{definition}
The Iwahori--Hecke algebra $\H(n)$ is defined as the algebra of $\B_n$--biinvariant elements in
$\mathbb C(\GL{n})$.
\end{definition}

 The following well-known proposition describes the structure of $\H(n)$.

\begin{proposition}
\label{Proposition_finite_Hecke}
 The algebra $\H(n)$ has dimension $n!$ and has a linear basis $s_\omega$ enumerated by
 permutation matrices $\omega$:
 $$
  s_\omega=\frac{1}{|\B_n|} \sum_{g\in \B_n \omega \B_n} e_g.
 $$
 As an algebra $\H(n)$ is generated by $n-1$ elements $s_{(i,i+1)}$ (where $(i,i+1)$ is elementary
 transposition permuting $i$ and $i+1$) subject to relations
\begin{enumerate}
\item $s_{(i,i+1)} s_{(j,j+1)}= s_{(j,j+1)} s_{(i,i+1)},\quad |i-j|>1,$
\item $s_{(k,k+1)} s_{(k+1,k+2)} s_{(k,k+1)} =s_{(k+1,k+2)} s_{(k,k+1)} s_{(k+1,k+2)},$
\item $s_{(k,k+1)}^2=(q-1) s_{(k,k+1)} + q s_e,$
\end{enumerate}
where $e$ is identical permutation. $s_e$ is the unit element in $\H(n)$.
\end{proposition}
\begin{proof}
 See \cite{I}, \cite{Bourbaki}.
\end{proof}

Let us embed $\H(n)$ into $\H(n+1)$ as a subalgebra spanned by first $n-1$ out of $n$ generators.
\begin{definition} The infinite--dimensional Iwahori--Hecke algebra $\H(\infty)$ is defined as the inductive limit of
$\H(n)$:
$$
 \H(\infty)=\varinjlim_{n\to\infty} \H(n) =\bigcup_n \H(n).
$$
\end{definition}

Note that $\H(\infty)$ is generated by countably many generators $s_{(i,i+1)}$ subject to the same
relations as in Proposition \ref{Proposition_finite_Hecke}.

Through the identification $\mathbb C(\GL{n})\simeq \A({\GLB})_n$ we can view $\H(n)$ as a
subalgebra of $\A(\GLB)_n$. The following proposition is straightforward.
\begin{proposition}
 The restriction of the embedding $i_n: \A({\GLB})_n \to \A({\GLB})_{n+1}$ on the subalgebra
 $\H(n)$
 coincides with above embedding $\H(n)\to \H(n+1)$, therefore, $\H(\infty)\subset\A(\GLB)$. $\H(\infty)$ coincides with
 subalgebra of $\B$--biinvariant functions in $\A(\GLB)\subset L_1(\GLB,\mu_{\GLB})$.
\end{proposition}

We want to define yet another important subalgebra of $\A(\GLB)$. Let $I_{\B}\in\A({\GLB})$ denote
the indicator function of $\B$ in the realization of $\A({\GLB})$ as a subalgebra of
$L_1(\GLB,\mu_{\GLB})$.

\begin{definition}
The \emph{unipotent subalgebra} $\A(\Uni)$ is defined as a two-sided ideal in $\A(\GLB)$ generated
by $I_{\B}$.
\end{definition}

 Our definitions imply that $\H(\infty)\subset \A(\Uni) \subset \A({\GLB})$.

\smallskip

We also note that the Bratteli diagram of $\A(\GLB)$ described in the previous section is a
disjoint union of countably many copies of the \emph{Young graph} with shifted gradings and
different labels of the vertices. Therefore, $\A(\GLB)$ is the direct sum of ideals corresponding
to the connected components of its Bratteli diagram. We remark that $\A(\Uni)$ is precisely the
component consisting of families $f\in\CY$ such that $f(u)=\emptyset$, unless $u=``x-1''$, see also
Proposition \ref{prop_def_unipotent} for a related fact.

\subsection{Classification of traces of $\A(\GLB)$} \label{Section_character_classif}

Although, we are not going to use it directly, but the following abstract statement
holds:

\begin{proposition}
The description of  traces of a locally semisimple algebra depends solely on its Bratelli diagram
without labels. In other words, if $\mathcal X$ and $\mathcal Y$ are two locally semisimple
algebras, whose Bratelli diagrams have the same sets of vertices and edges but, perhaps, different
labels of vertices, then there is a canonical correspondence between their  traces.
\end{proposition}
\begin{proof}[Sketch of the proof]
 This follows from the identification of  traces with \emph{harmonic functions} or
 \emph{coherent systems} on the Bratelli diagram of the algebra, see \cite{VK_AF}, \cite{VK_Long}, \cite{Kerov_book}
 for more details. The key idea here is that branching of  traces does not depend on labels, for $\GLB$ this can be
seen in Theorem \ref{proposition_branching_of_characters_GLB}.
\end{proof}

In order to state the classification theorem for  traces of $\A({\GLB})$ we need to introduce some
additional notations.

Let $f\in\CY$ be a family of Young diagrams. We call the set
$$
\{x\in\C\mid f(x)\ne \emptyset\}
$$
the support of $f$ and denote it $supp(f)$. If $f$ and $g$ are two families of Young diagrams with
disjoint supports, then $f+g$ stays for the following family:
$$
 (f+g)(x)=\begin{cases} f(x),\, x\in supp(f),\\ g(x),\, x\in supp(g),\\ \emptyset,\text{
 otherwise.}\end{cases}
$$
This operation corresponds to the \emph{parabolic induction} of representations of $\GL{n}$ (see
e.g.\ \cite{Green}, \cite[Chapter III]{Zel} and \cite[Section IV.3]{M})

Let $\Lambda$ be the algebra of symmetric functions in variables $x_1,x_2,\dots$ (see e.g.\
\cite{M} for all the definitions). We intensively use various generators of this algebra, namely,
elementary symmetric functions $e_n$, complete symmetric functions $h_n$ and power sums $p_k$:
$$
 p_k=\sum_k x_i^k.
$$
We also use \emph{Schur symmetric functions} $s_\lambda$, $\lambda\in\Y$ which form a linear basis
in $\Lambda$.

A \emph{specialization} $\Phi$ of $\Lambda$ is an algebra homomorphism:
$$
 \Phi: \Lambda \to \mathbb C.
$$
Note that any specialization of $\Lambda$ is uniquely defined by its values on $p_k$. In what
follows we write the arguments of specializations in square brackets $\Phi[\cdot]$.

Let $\alpha=\{\alpha_i\}$ and $\beta=\{\beta_i\}$, $i=1,2,3\dots$ be two weakly decreasing
sequences of non-negative real numbers such that
\begin{equation}
\label{eq_summable_sequence}
 \sum_{i=1}^{\infty} (\alpha_i+\beta_i) \le \gamma <\infty.
\end{equation}

\begin{definition}
For any two sequences $\alpha$ and $\beta$ of non-negative reals and number $\gamma$ satisfying
\eqref{eq_summable_sequence} we define the specialization $\Sp_{\alpha,\beta,\gamma}$ through its
values on the generators $p_k$ of $\Lambda$
$$
 \Sp_{\alpha,\beta,\gamma}[p_1]=\gamma,\quad \Sp_{\alpha,\beta,\gamma}[p_k]=\sum_i\alpha_i^k+ (-1)^{k-1}\sum_i \beta_i^k.
$$
\end{definition}
\noindent {\bf Remark. } Note that if $\beta_i=0$ and $\sum_i \alpha_i=\gamma$, then the
specialization $\Sp_{\alpha,\beta,\gamma}$ boils down to the substitution of numbers $\alpha_i$ in
place of formal variables $x_i$.

\begin{definition} $\CY'\subset \CY$ is the set of families $f$ of Young diagrams such that
$f(``x-1")=\emptyset$.
\end{definition}

\begin{definition}
 $\Omega(\GLB)$ is defined as the set of triplets $(\alpha,\beta,f)$, where
$\alpha=\{\alpha_i\}$ and $\beta=\{\beta_i\}$, $i=1,2,3\dots$ are two weakly decreasing sequences
of non-negative real numbers satisfying \eqref{eq_summable_sequence} for $\gamma=1$ and
$f\in\CY'$.
\end{definition}
\begin{definition}
For $\omega\in\Omega({\GLB})$ we define a  trace $\chi^{\omega}$ of $\A({\GLB})$ as follows. For
$g\in\GL{n}$ we have
 $
  \chi^{\omega}(I^\GLB_g)=0
 $
 if $n<|f|$, otherwise,
 \begin{equation}
 \label{eq_extreme_GLB}
  \chi^{\omega}(I^\GLB_g)=\sum_{\lambda\in \Y_{n-|f|}} \chi^{f+E_1(\lambda)}(I^\GLB_g) \Sp_{\alpha,\beta,1}[s_\lambda],
 \end{equation}
 where $E_1(\lambda)$ is a function from $\CY_{n-|f|}$ taking value $\lambda$ in $``x-1"$ and taking value
 $\emptyset$ in all other points. $\chi^{f+E_1(\lambda)}$, as and above, stays for the matrix trace of
 the irreducible representation of $\A(\GLB)_n$ ($\GL{n}$) indexed by $f+E_1(\lambda)$.
\end{definition}

\begin{theorem}[Classification theorem for finite traces of $\A(\GLB)$]
\label{theorem_characters_of_GLB}
 The extreme rays of the set of  traces of $\A(\GLB)$ are parameterized by elements of $\Omega(\GLB)$. For
 $\omega=(\alpha,\beta,f)\in \Omega(\GLB)$ the corresponding ray is $\mathbb R_+
 \chi^{\omega}(\cdot)$.
\end{theorem}
\begin{proof}
 For a family $f\in\CY'$ let $\CY^{(f)}\subset\CY$ denote the set of families $h\in\CY$ such that
 $h(u)=f(u)$ for all $u\in\mathcal C\setminus\{``x-1"\}$.

 Moreover, for a family $f\in\CY'$ let $\Upsilon^f$ denote the convex cone of  traces $\chi$ of $\A(\GLB)$ such that
 such that for $n<|f|$ the restriction $\chi\rule[-2.3mm]{.4pt}{4mm}_{\,\A(\GLB)_{n}}$ vanishes and for $n\ge |f|$
 in the decomposition (see Corollary \ref{cor_decomposition_of_restriction})
$$
 \chi \rule[-2.3mm]{.4pt}{4mm}_{\,\A(\GLB)_{n}} =\sum_{h\in\CY_n} c(h) \chi^h(\cdot).
$$
$c(h)=0$ unless $h\in\CY^{(f)}$. Let $\Upsilon^\emptyset$ denote the set $\Upsilon^f$ for $f$
being the empty family. We claim that for any $f\in\CY'$ the convex cone $\Upsilon^f$ is affine
isomorphic to $\Upsilon^\emptyset$. The isomorphism
$$
\Phi^f: \Upsilon^\emptyset \to \Upsilon^f
$$
is given be the following formula. If $\chi\in\Upsilon^\emptyset$ is such that
$$
 \chi \rule[-2.3mm]{.4pt}{4mm}_{\,\A(\GLB)_{n}} =\sum_{h\in\CY_n\bigcap\CY^{(f)} } c(h) \chi^h(\cdot),
$$
then
$$
 \Phi^f(\chi) \rule[-2.3mm]{.4pt}{4mm}_{\,\A(\GLB)_{n+|f|}} =\sum_{h\in\CY_n\bigcap\CY^{(f)} } c(h)
 \chi^{h+f}(\cdot).
$$

Theorem \ref{proposition_branching_of_characters_GLB} implies that the branching of  traces from
$\Upsilon^\emptyset$ with respect to restriction on subalgebras $\A(\GLB)_n$ is the same as
branching of the characters of symmetric groups, cf.\ \cite{Sagan}, \cite{Kerov_book}. Therefore,
$\Upsilon^\emptyset$
 is isomorphic to the set of characters of the infinite symmetric group
$S(\infty)$, see \cite{VK_Long}, \cite{Kerov_book}. The latter characters were classified by Thoma
\cite{Thoma1}, see also \cite{VK_S}. Thoma's theorem implies that the extreme rays of
$\Upsilon^\emptyset$ are parameterized by pairs $\alpha=\{\alpha_i\}$ and $\beta=\{\beta_i\}$,
$i=1,2,3\dots$ of weakly decreasing sequences of non-negative real numbers satisfying
\eqref{eq_summable_sequence} with $\gamma=1$. The ray corresponding to a pair $(\alpha,\beta)$ is
spanned by the character $\chi^{\alpha,\beta}$ such that for $g\in\GL{n}$ we have
 $$
  \chi^{\alpha,\beta}(I^\GLB_g)=\sum_{\lambda\in \Y_n} \chi^{\lambda}(I^\GLB_g) \Sp_{\alpha,\beta,1}[s_\lambda],
 $$

We conclude that for $f\in\CY'$ the extreme rays of $\Upsilon^f$ are parameterized by pairs
$(\alpha,\beta)$ and are given by the formula \eqref{eq_extreme_GLB} for the triplet
$(\alpha,\beta,f)$.

It remains to prove that every extreme ray of the set of  traces of $\A(\GLB)$ is an extreme ray
of one of the sets $\Upsilon^f$. Indeed, let $\chi$ be a  trace of $\A(\GLB)$.
 We claim that there exists a unique decomposition of $\chi$ into the sum
 $$
  \chi=\sum_{f\in\CY'} \chi^{(f)},\quad \chi^{(f)} \in \Upsilon^f
 $$
To prove the claim consider the restrictions $\chi^{(f)}\rule[-2.3mm]{.4pt}{4mm}_{\,\A(\GLB)_{n}}$
for which the existence and uniqueness of such decomposition is immediate. This finishes the
proof.
\end{proof}

\section{Unipotent traces and their values}
\label{Section_unipotent_values}

Recall that an irreducible representation of $\GL{n}$ is said to be \emph{unipotent} (see e.g.
\cite{Steinberg}, \cite{James_unipotent}) if it contains a non-zero $\B_n$-invariant vector. (Here
$\B_n\subset\GL{n}$ is the subgroup of upper-triangular matrices.) In the above parameterization
of irreducible representation of $\GL{n}$ by the families of Young diagrams, unipotent
representations $\pi^f$ are precisely those for which $f(p)=\emptyset$ if $p\ne``x-1"$.

\begin{proposition}
\label{prop_def_unipotent} Let $\chi^{\omega}$, $\omega\in\Omega(\GLB)$ be an indecomposable trace
of $\A(\GLB)$. The following conditions are equivalent:
\begin{enumerate}
\item For every $n$ the restriction of
$\chi^{\omega}$ to $\A(\GLB)_n$ is a linear combination of matrix traces of irreducible unipotent
representations of $\GL{n}$,
\item Restriction of $\chi^{\omega}$ on $\H(\infty)$ is non-zero,
\item Restriction of $\chi^{\omega}$ on $\A(\Uni)$ is non-zero,
\item $\omega=(\alpha,\beta,f)$ with $f\equiv
\emptyset$.
\end{enumerate}
\end{proposition}
\begin{proof}
 Equivalence of properties $(1)$ and $(4)$ is a corollary of Theorem
 \ref{theorem_characters_of_GLB}. Next, normalized indicator function of the Borel subgroup
 $\B_n\subset\GL{n}$ is a unit element of $\H(n)$. In the same time in a unipotent representation
 of $\GL{n}$ it acts as a projection on the set of $\B_n$--invariant vectors, while in any other
 representation it acts as zero. Therefore, the value of the matrix trace of a unipotent
 representation of $\GL{n}$ on this indicator function is positive, and $(1)$ implies $(2)$. Since
 $\H(\infty)\subset \A(\Uni)$, the property $(3)$ follows from $(2)$. Finally, since
 $\A(\Uni)$ is spanned by  the indicator function of $\B$ and this indicator function vanishes
 in any non-unipotent representation, every element of $\A(\Uni)$ acts as zero in any
 non-unipotent representation. Therefore, the value of a non-unipotent character of $\GL{n}$ on an
 element of $\A(\Uni)$ is zero and $(1)$ follows from $(3)$.
\end{proof}

\begin{definition} An indecomposable trace of $\A(\GLB)$ satisfying conditions
of Proposition \ref{prop_def_unipotent} is called \emph{unipotent}.
\end{definition}

 In this section we find a number of remarkable properties of unipotent
traces which give a relatively simple procedure for the computation of their values on arbitrary
elements of $\A(\GLB)$. Let us sketch all these properties together first. Since $\A(\GLB)_n$ is
isomorphic to the group algebra of $\GL{n}$, the  traces can be viewed as functions on matrices
from $\GL{n}$ for various $n$. Such function is central, i.e.\ its values depend on the matrix
through its Jordan normal form. One property of these functions is their multiplicativity which
expresses the value on arbitrary Jordan normal forms as product of values on single block Jordan
forms. Another property is a simple relation between values on the Jordan blocks with eigenvalue
$1$ and on Jordan blocks with arbitrary other eigenvalues. Final component is an expression for
the values on the Jordan blocks with eigenvalue $1$ in terms of specializations of \emph{modified
Hall-Littlewood polynomials}.

\smallskip

As for the restriction of unipotent trace on Hecke algebra $\H(\infty)\subset \A(\GLB)$, in this
section we identify them with extreme  traces of $\H(\infty)$ found in \cite{VK_Hecke}, see also
\cite[Section 7]{Me}, which also gives a formula for their values.

\subsection{Values of unipotent traces: formulations}

\begin{theorem}[Multiplicativity theorem for unipotent traces]
 \label{Theorem_multiplicativity_of_characters} Let $\chi^{\omega}$ be an extreme unipotent trace. For $g\in\GL{n}$ let $\chi(g)$ denote the value of the
 restriction of $\chi^{\omega}$ to $\A(\GLB)_n\simeq\mathbb C(\GL{n})$ on the element $e_g\in\mathbb
 C(\GL{n})$. Suppose that $a\in\GL{n}$ and $b\in\GL{m}$  are two matrices with coprime characteristic polynomials, then
 $$
  \chi^{\omega}(a)\chi^{\omega}(b)=\chi^{\omega}(a\odot b),
 $$
 where $a \odot b\in\GL{n+m}$ is the block-diagonal matrix with blocks $a$ and $b$.
\end{theorem}
\noindent {\bf Remark.} A very similar multiplicativity property holds for the extreme characters
of the infinite symmetric group $S(\infty)$ (see \cite{Thoma1}) and infinite-dimensional unitary
group $U(\infty)$ (see \cite{Vo}). This seems to be a general infinite-dimensional phenomenon, cf.\ \cite{Vo_multi},
\cite{VK_ring}, \cite{Olsh_Semi}.


\smallskip

The values of the unipotent  traces on various conjugacy classes can be computed in terms of
 specializations of symmetric functions.

Let $P_\mu(x_1,x_2,\dots;q^{-1})$ and $Q_\mu(x_1,x_2,\dots;q^{-1})$  denote the Hall-Littlewood
$P$ and $Q$ polynomials with parameter $q^{-1}$ in variables $x_1,\dots$ labeled by a Young
diagram $\mu$, see \cite[Chapter III]{M}. Let $\mathbb M_q$ denote the endomorphism of the algebra
of symmetric functions
$$
 \mathbb M_q: \Lambda \to \Lambda
$$
given on the Newton power sums $p_k$ by the formula
$$
 p_k\to \frac{1}{1-q^{-k}} p_k, \quad k\ge 0.
$$
Denote
$$
 \widetilde P_\mu= \mathbb M_q P_\mu, \quad \widetilde Q_\mu= \mathbb M_q Q_\mu.
$$
The symmetric functions $\widetilde P_\mu$ and $\widetilde Q_\mu$ are known as \emph{modified
Hall-Littlewood polynomials.}

\begin{theorem}
\label{theorem_character_at_unipotent_class}
 Let $\chi^{\omega}$ be an extreme unipotent trace. For $g\in\GL{n}$ let $\chi(g)$ denote the value of the
 restriction of $\chi^{\omega}$ to $\A(\GLB)_n\simeq\mathbb C(\GL{n})$ on the element $e_g\in\mathbb
 C(\GL{n})$. Suppose that the characteristic polynomial of $g$ is $(x-1)^n$ and the conjugacy
 class (i.e.\ Jordan Normal form of $g$) is encoded by the Young diagram $\lambda$ with $n$ boxes. Then
 $$\chi(g) = q^{n(\lambda)}\Sp_{\alpha,\beta,1}\left[\widetilde Q_\lambda\right].$$
\end{theorem}

Now let $Pl_n$ be the  endomorphism of the algebra $\Lambda$ defined through
$$
 Pl_n:\Lambda\to\Lambda,\quad Pl_n(p_i)=p_{ni}
$$
This is a particular case of \emph{plethysm} morphism, see \cite[Section I.8]{M}. Note that $Pl_n$
maps $m_\lambda$ to $m_{n\lambda}$.

Set
$$
 \Sp^n_{\alpha,\beta,1}:=\Sp_{\alpha,\beta,1}\circ Pl_n.
$$
Observe that
$$
 \Sp^n_{\alpha,\beta,1}\left[f\right]=\Sp_{\alpha^n,-(-\beta)^n,1}\left[f\right],
$$
where $\alpha^n=((\alpha_1)^n,(\alpha_2)^n,\dots)$,
$-(-\beta)^n=(-(-\beta_1)^n,-(-\beta_2)^n,\dots)$.

Theorem \ref{theorem_character_at_unipotent_class} can be generalized as follows.

\begin{theorem}
\label{theorem_character_at_simple_class}
 In the settings of Theorem \ref{theorem_character_at_unipotent_class} suppose that $n=km$,
 the characteristic polynomial of $g$ is $u^m$, where $u$ is an irreducible (over $\F$) polynomial of degreee $k$ and the conjugacy
 class (i.e.\ Jordan Normal form of $g$) is given by the Young diagram $\lambda$ with $m$ boxes. Then
 $$\chi(g) = q^{kn(\lambda)}\Sp^k_{\alpha,\beta,1}\left[\widetilde Q_\lambda(\cdot; q^{-k})\right]$$
\end{theorem}

Clearly, combining Theorem \ref{theorem_character_at_simple_class} and Theorem
\ref{Theorem_multiplicativity_of_characters} we get the formula for values of the unipotent
 traces on arbitrary conjugacy classes.

\subsection{Values of unipotent traces: proofs}

To give a proof of Theorem \ref{Theorem_multiplicativity_of_characters} we need some facts about
the Hopf algebra related to the representations of $\GL{n}$. We follow \cite{Zel} and \cite{M}
here.

Let $D_n$ denote the the space of central (i.e.\ class) complex functions on $\GL{n}$. This is a
finite dimensional vector space with basis of the characteristic functions of the conjugacy
classes of $\GL{n}$. The latter are parameterized by the elements of $\CY_n$.

Denote
$$
 D=\oplus_{n\ge 0} D_n.
$$
The vector space $D$ is a Hopf algebra with multiplication and comultiplication given by the
operations of parabolic induction and parabolic restriction, respectively, see e.g.\ \cite[Section
8, Chapter III]{Zel} or  \cite[Section IV]{M}.

For a family $f\in\CY_n$ let $Cl_f\in D_n$ denote the indicator function of the conjugacy class in
$\GL{n}$ corresponding to $f$. The definition of the multiplication in $D$ implies that for
disjoint $f_1,f_2$ we have
\begin{equation}
\label{eq_multiplicativity_of_classes}
 Cl_{f_1} \cdot  Cl_{f_2} = Cl_{f_1+f_2}
\end{equation}

Let $R_n\subset D_n$ denote the $\mathbb Z$--module spanned by the characters of irreducible
representations of $\GL{n}$ and $R=\oplus_{n\ge 0} R_n$. Then $D=R\otimes_{\mathbb Z} \mathbb C$.
$R$ is a Hopf subalgebra of $D$, moreover, $R$ is a \emph{PSH-alebra} in the terminology of
\cite{Zel}.

\begin{proposition}
\label{proposition_Hopf_summary} We have
\begin{equation}
\label{eq_decomposition_into_tensor_product}
 R=\bigotimes\limits_{c\in\mathcal C} R^c,\quad  D=\bigotimes\limits_{c\in\mathcal C} D^c,
\end{equation}
where $\bigotimes$ means the tensor product of Hopf algebras, $R^c$ is $\mathbb Z$--module spanned
by the characters of irreducible representations $\chi^f$ of $\GL{n}$ such that $f(u)=\emptyset$
unless $u=c$. $R^c$ is a Hopf subalgebra of $R$, elements $c\in\mathcal C$ enumerate the so-called
\emph{cuspidal} irreducible representations of $\GL{n}$, which are in bijections with elements of
$\C$ and $D^c=R^c\otimes_{\mathbb Z} \mathbb C$. Each $R^c$ is isomorphic to $\Lambda$, under this
identification $\chi^f\in R^c$ corresponds to $s_{f(c)}$.
\end{proposition}
\begin{proof} See \cite[9.3]{Zel}. \end{proof}

%

Next we describe how  the  traces of $\A(\GLB)$ are related to algebra $D$. Let $p^\GLB$ denote a
distinguished degree one element of $D$ which is the sum of characters of all $q-1$ irreducible
representations of $\GL{1}$. Note that the support of $p^{\GLB}$ as a function on $\GL{1}$ is the
unit element.

Let $\Xi^\GLB$ denote the convex cone of linear functionals on $D$ satisfying:
$$
 \xi: D\to\mathbb C
$$
\begin{enumerate}
\item $\xi[u\cdot p^\GLB]=\xi[u]$, for every $u\in D$,
\item $\xi[\chi]\ge 0$ for every $n$ and every character $\chi\in D_n$ of an irreducible representation
of $\GL{n}$.
\end{enumerate}

Let $\phi\to \xi^\phi$ denote the map from the set of  traces of $\A(\GLB)$ to $\Xi^\GLB$ given
by:
$$
 \xi^\phi[\chi^f]=c(f),
$$
where $\chi^f\in\ D_n$ is a character of irreducible representation of $\GL{n}$ indexed by
$f\in\CY_n$ and $c(f)$ is the coefficient in the decomposition
$$
 \phi \rule[-2.3mm]{.4pt}{4mm}_{\,\A(\GLB)_{n}}(\cdot) =\sum_{f\in\CY_n} c(f) \chi^f(\cdot).
$$

Comparing the definitions of the set $\Xi^\GLB$ and Proposition \ref{proposition_Hopf_summary}
with Theorem \ref{proposition_branching_of_characters_GLB} we conclude that the map $\phi\to
\xi^\phi$ gives a bijection between the set of  traces of $\A(\GLB)$ and $\Xi^\GLB$.

Moreover, note that for $g\in\GL{n}$ belonging to a conjugacy class $f$ we have (as follows from
the definitions)
\begin{equation}
\label{eq_value_of_character_through_functional}
 \phi(e_g)=\xi^{\phi}[Cl_f]
\end{equation}
As and above we use the identification $\A(\GLB)_n\simeq \mathbb C(\GL{n})$ here.

Now we can prove Theorem \ref{Theorem_multiplicativity_of_characters}.
\begin{proof}[Proof of Theorem \ref{Theorem_multiplicativity_of_characters}]

Let $\widetilde\chi$ (which is $\chi^\omega$ for some $\omega\in\Omega(\GLB)$) be a unipotent
extreme  trace of $A(\GLB)$. For $b\in R^{``x-1"}$ the values of $\xi^{\widetilde\chi}[b]$ were
computed in Theorem \ref{theorem_characters_of_GLB}. In particular, since specializations
$\Sp_{\alpha,\beta,\gamma}$ are algebra homomorphisms by their definition and $R^{``x-1''}\simeq
\Lambda$, we have for $b_1,b_2\in R^{``x-1"}$
\begin{equation}
 \label{eq_multiplicativity}
 \xi^{\widetilde\chi}[b_1 b_2]=\xi^{\widetilde\chi}[b_1] \xi^{\widetilde\chi}[b_2].
\end{equation}
 We claim that \eqref{eq_decomposition_into_tensor_product} actually holds for general $b_1, b_2 \in
 D$. To prove this claim note that both sides of \eqref{eq_multiplicativity} are bilinear.
 Therefore, it enough to check this property for $b_1,b_2$ belonging to some linear basis of $D$.
 Let us choose the basis $b^f$ enumerated by $f\in \CY$ and given by
 $$
  b^f:=\chi^f = \prod_{c\in\mathcal C} \chi^{f_c},
 $$
 where $\chi^{f_c}\in R^c$ is the character of the irreducible representation of $\GL{|f(c)|}$
 corresponding to the family $f_c$ defined as a unique family such that $f_c(c)=f(c)$ and
 $f_c(u)=\emptyset$ for $u\ne c$; $\chi^f$ as and above is the corresponding  character  of the irreducible representation
 of $\GL{|f|}$. It remains to observe that by the definition of the unipotent
 character
 $$
  \xi^{\widetilde\chi}[b^f]=0
 $$
 unless $f(u)=0$ for any $u\ne``x-1"$.

\smallskip

Now suppose that $a\in\GL{n}$ and $b\in\GL{m}$  are two matrices with coprime characteristic
polynomials belonging to conjugacy classes parameterized by $f\in\CY_{n}$ and $h\in\CY_{m}$,
respectively. Then the families $f$ and $h$ are disjoint, furthermore, $a\circ b$ belongs to the
conjugacy class parameterized by $f+h$. Therefore, using \eqref{eq_multiplicativity},
\eqref{eq_multiplicativity_of_classes} and \eqref{eq_value_of_character_through_functional} we
obtain
$$
 \chi^{\omega}(a\circ b)=\xi^{\chi^\omega}[Cl_{f+h}]=\xi^{\chi^\omega}[Cl_{f}\cdot Cl_h]=
 \xi^{\chi^\omega}[Cl_{f}]\xi^{\chi^\omega}[Cl_h]=\chi^{\omega}(a)\chi^{\omega}(b). \qedhere
$$
\end{proof}

\begin{proof}[Proof of Theorem \ref{theorem_character_at_unipotent_class}]

If $\chi^\mu$ is the character of the unipotent representation of $\GL{n}$ indexed by $\mu$, then
as follows e.g.\ from the results of \cite[Chapter IV]{M}
$$
 \chi^\mu(g)=q^{n(\lambda)} K_{\mu,\lambda}^{q^{-1}},
$$
where $K_{\mu,\lambda}^{q^{-1}}$ is the $q^{-1}$ Kostka number defined as the coefficient in the
decomposition of Schur polynomials into the sum of Hall-Littlewood polynomials
\begin{equation}
\label{eq_qKostka}
 s_\mu(x_1,x_2,\dots)=\sum_\lambda K_{\mu,\lambda}^{q^{-1}} P_\lambda(x_1,x_2,\dots;q^{-1}).
\end{equation}
Therefore,
\begin{equation}
\label{eq_characters_through_Kostka}
 \chi^\omega(g)=q^{n(\lambda)}\sum_\mu K_{\mu,\lambda}^{q^{-1}} \Sp_{\alpha,\beta,1}[s_\lambda].
\end{equation}
It remains to prove that
\begin{equation}
\label{eq_Q_mod_into_Schur}
 \widehat Q_\lambda= \sum_\mu K_{\mu,\lambda}^{q^{-1}} s_\mu.
\end{equation}
Indeed, the Cauchy  identity for Hall-littlewood  polynomials (see \cite[Section III.4]{M}) yields
\begin{equation}
\label{eq_HL_Cauchy}
 \sum_{\nu\in\Y} P_\nu(x_1,x_2,\dots;q^{-1}) Q_\nu(y_1,y_2,\dots;q^{-1})u = \prod_{i,j}
 (1-q^{-1}x_i y_j)(1-x_i y_j)
\end{equation}
Applying the map $M_q$ with respect to the variables $y_1,y_2,\dots$ in the identity
\eqref{eq_HL_Cauchy} we arrive at
\begin{equation}
\label{eq_HL_Cauchy_modified}
 \sum_{\nu\in\Y} P_\nu(x_1,x_2,\dots;q^{-1}) \widetilde Q_\nu(y_1,y_2,\dots;q^{-1})u = \prod_{i,j}
 (1-x_i y_j)^{-1}
\end{equation}
In the same time the Cauchy Identity for Schur Polynomials (see \cite[Section I.4]{M}) yields
\begin{equation}
\label{eq_Schur_Cauchy}
 \sum_{\nu\in\Y} s_\nu(x_1,x_2,\dots;q^{-1}) s_\nu(y_1,y_2,\dots;q^{-1})u = \prod_{i,j}
 (1-x_i y_j)^{-1}
\end{equation}
Combining \eqref{eq_HL_Cauchy_modified}, \eqref{eq_Schur_Cauchy} and \eqref{eq_qKostka} we arrive
at \eqref{eq_Q_mod_into_Schur}.
\end{proof}

We need some preparations to prove Theorem \ref{theorem_character_at_simple_class}. In order to
connect the values of unipotent  traces on unipotent and on more general conjugacy classes it is
convenient to work not with irreducible representations of $\GL{n}$, but with representations
induced from the parabolic subgroups.

More precisely, let $\mu$ be a Young diagram with $n$ boxes and let $fl_\mu$ denote the set of all
flags of subspaces
$$
 V_1\subset V_2\subset\dots\subset V_r
$$
of $n$-dimensional vector space $\F^n$, such that $\dim V_k=\mu_1+\dots+\mu_k$. Here $r$ is the
number of nonempty rows in $\mu$. $\GL{n}$ naturally acts in $fl_\mu$. Let $\psi^q_\mu$ denote the
character of the corresponding representation of $\GL{n}$ in $\mathbb{C}( fl_\mu)$. Clearly,
$\psi_\mu(g)$ is equal to the number of flags in $fl_\mu$ which $g$ fixes.

First, we claim that if the conjugacy class of $g$ is given by a family $f\in\CY_n$ such that
$f(p)=\emptyset$ for all but one linear polynomial $p$, then the value of $\psi_\mu(g)$ does not
depend on this $p$ (but, of course, depends on the Young diagram $f(p)$). Indeed, if $g_1$ and
$g_2$ are two such matrices, then one can be obtained from another by conjugation and addition of
a scalar matrix. Conjugation does not change the character. The addition of a scalar matrix does
not change the set of invariant subspaces of an operator, thus, also does not change the
character.

Next suppose that $n=mk$ and let $p(x)$ be an irreducible polynomial of degree $k$. Suppose that
the conjugacy class of $g\in\GL{n}$ is given by a family $f$, $f(\cdot)$ is equal to empty set
everywhere except at $p$ and $f(p)=\lambda$. This implies that $|\lambda|=m$. Suppose also that
$g'_y\in\mathbb{GL}(m,q^k)$ (note that the number of elements in the field changed!) is in a
conjugacy class $f'$ such that $f(p')=\lambda$ for linear polynomial $p'(x)="x-y"$ (here
$y\in\mathbb F_{q^k}^*$). Our next aim is to prove that
\begin{equation}
\label{eq_induced_character_formula}
\psi_\mu^q(g)=\begin{cases} \psi_\nu^{q^k}(g'_y),\text{ if } \mu_i= k\nu_i\text{ for all }i,\\
 0,\text{ otherwise.}\end{cases}
\end{equation}
Note that (as we have shown above) the right side of \eqref{eq_induced_character_formula},
actually, does not depend on $y$.

Conjugating the matrix, if necessary, we can assume that $g$ is made out of $k\times k$ blocks. On
the main diagonal all the blocks are $Mat(p)$, where
$$
 Mat(p)=\begin{pmatrix} 0& 0 & 0 & \dots & -p_0 \\ 1& 0 & 0 &  \dots &-p_1 \\ & & \dots & &\\  0 &\dots& 1 & 0 & -p_{k-2}\\ 0&   \dots&0&1& -p_{k-1}\end{pmatrix}
$$
is the companion matrix of the polynomial
$$p(x)=p_0+p_1 x+\dots +p_{k-1} x^{k-1} + x^k.$$
Below the diagonal all blocks are zero. Directly above the main diagonal blocks are either
$k\times k$ identity matrices or zeros. Identity matrices form (diagonal) groups of lengths (from
top to bottom) $\lambda_1$, $\lambda_2$, \dots.

The matrix $g'$ can be assumed to have a similar structure (but without any blocks), e.g.\
$$
 g'=\begin{pmatrix} y& 1& 0\\0& y& 1 &0 \\0& 0& y& 0& 0\\0& 0 & 0 & y & 1 & 0\\ & & \dots \\ 0&\dots & & & &0& y\end{pmatrix}
$$
More formally, $g'$ has $y$s on the main diagonal, zeros everywhere below the main diagonal and
above the second (i.e.\ the one on top of the main diagonal) and $1$s in the second diagonal
forming groups divided by zeros. In the above example the length of the first group, which is
$\lambda_1$, equals $2$.

Now let us view the $n$-dimensional space $\F^n$ as $(\F^k)^m$. Let us identify $\F^k$ with
$\mathbb F_{q^k}$. The main step in proving \eqref{eq_induced_character_formula} is the following
lemma.
\begin{lemma}
 \label{lemma_invariant_subspaces} There exists $y\in\mathbb F_{q^k}$ such that if
  $V$ is a $\F$-linear subspace of $(\F^k)^m$ invariant under $g$, then $V$ is a $\mathbb
  F_{q^k}$-linear subspace invariant under $g'_y$ and vice versa.
\end{lemma}
\begin{proof}
 Let $h$ be $n\times n$ matrix made out of $m$ $k\times k$ blocks on the diagonal, each block
 is $Mat(p)$. If we set $Q=q^t$ with large enough $t$, then $(g-h)^Q=0$. Note that the matrices
 $g$ and $h$ commute, therefore $(g-h)^Q=g^Q+(-h)^Q$. We conclude that $V$ is invariant under
 $h^Q$.

 Now let us consider the field $\F[x]/p(x)\simeq\mathbb F_{q^k}$. Note that $Mat(p)$ is the matrix of
 the operator of multiplication by $x$ in the basis $1,x,\dots,x^{k-1}$. Since $x$ now can be
 viewed as an element of the field with $q^k$ elements, we conclude that $Mat(p)^{q^k}=Mat(p)$.
 Therefore, it is possible to choose large $t$ so that $Mat(p)^Q=Mat(p)$. Thus, $h^Q=h$ and $V$ is
 invariant under $h$. Any element of $\F[x]/p\simeq\mathbb F_{q^k}$ is a polynomial in $x$,
 therefore, $V$ is invariant under the multiplicative group of $\mathbb F_{q^k}$. In other words,
 $V$ is a $\mathbb F_{q^k}$-linear subspace. Now identifying $x\in\F[x]/p(x)$ with $y\in\mathbb
 F_{q^k}$ we see that $V$ is invariant under the $g'_y$.

 In the reverse direction the statement is immediate.
\end{proof}

Now \eqref{eq_induced_character_formula} becomes straightforward. Indeed, the left side of
\eqref{eq_induced_character_formula} equals the number of $g$-invariant flags. If $k$ does not
divide $\mu_i$ for some $i$, then Lemma \ref{lemma_invariant_subspaces} implies that there are
simply no invariant flags. And if $\mu_i=k\nu_i$ for all $k$, then the flags from $fl_\mu$ are
identified with flags from $fl_\nu$ over bigger field $\mathbb F_{q^k}$ and we arrive at the right
side of \eqref{eq_induced_character_formula}.

\begin{proof}[Proof of Theorem \ref{theorem_character_at_simple_class}]

Now we can deduce the formula for the values of the unipotent traces. Decompose $\chi^\omega$ into
the sum of the characters $\psi_\mu^q$:
$$
 \chi^\omega=\sum_{\mu\in \Y_n} c_\mu \psi_\mu^q
$$
In order to calculate the coefficients $c_\mu$ we recall the decomposition of the characters of
the irreducible unipotent representations of $\GL{n}$ into the sum of $\psi_\mu^q$. We have
\begin{equation}
\label{eq_Kostka}
 \chi^\lambda=\sum_\mu K_{\lambda,\mu} \psi_\mu^q,
\end{equation}
where the coefficients $K_{\lambda,\mu}$ are Kostka numbers and do not depend on $q$ (see
\cite{Steinberg} and also \cite[Section I.6]{M} and references therein). They can be defined via
the relations in the algebra of symmetric functions $\Lambda$:
$$
 m_\mu=\sum_\lambda K_{\lambda,\mu} s_\lambda
$$
where $m_\mu$ is the monomial symmetric function indexed by $\mu$. We have
\begin{multline*}
 \chi^\omega=\sum_{\lambda\in \Y_n} \chi^\lambda \Sp_{\alpha,\beta,1}[s_\lambda] =
 \sum_{\mu\in \Y_n} \psi_\mu^q \left(\sum_{\lambda\in\Y_n} K_{\lambda,\mu} \Sp_{\alpha,\beta,1}[s_\lambda]\right)\\=\sum_{\mu\in \Y_n} \Sp_{\alpha,\beta,1}[m_\mu]
 \psi_\mu^q
\end{multline*}
Evaluating in $g$ and using \eqref{eq_induced_character_formula} we get
$$
 \chi^\omega(g)=\sum_{\nu\in \Y_m} \Sp_{\alpha,\beta,1} [m_{k\nu}] \psi_\nu^{q^k}(g')=
 \sum_{\nu\in \Y_m} \Sp^k_{\alpha,\beta,1} [m_{\nu}] \psi_\nu^{q^k}(g')
$$
Converting back into the sum of irreducible unipotent characters we get
$$
\chi^\omega(g)=\sum_{\lambda\in \Y_m} \Sp^k_{\alpha,\beta,1} [s_\lambda] \chi^\lambda(g').
$$
Now the application of Theorem \ref{theorem_character_at_unipotent_class} (with $q$ replaced by
$q^k$) completes the proof.
\end{proof}

\subsection{Restriction of unipotent traces to Iwahori--Hecke algebra}

In this section we explain what happens when one restricts unipotent trace of $\A(\GLB)$ on
$\H(\infty)$.

First, we recall a classical theorem relating representations of $\H(n)$ and $\GL{n}$.

\begin{proposition}
\label{prop_rep_GL_H} Both irreducible representations of $\H(n)$ and unipotent irreducible
representations of $\GL{n}$ are parameterized by the set $\Y_n$ of Young diagrams with $n$ boxes.
The representation of $\H(n)$ indexed by $\lambda$ coincides with the restriction of the
representation of $\GL{n}$ indexed by $\lambda$ on the set of $\B_n$--invariant vectors and on the
subalgebra $\H(n)\subset \mathbb C(\GL{n})$. In particular the restriction of the conventional
character (matrix trace) of the irreducible unipotent representation of $\GL{n}$ on $\H(n)$ is the
character of the corresponding irreducible representation of $\H(n)$.
\end{proposition}
\begin{proof} See e.g.\
\cite{CF}.
\end{proof}

The  traces of infinite Hecke algebra $\H(\infty)$ were first classified in \cite{VK_Hecke},
recently they were also studied in \cite{Me}. From these articles the following result is known.
\begin{proposition}
\label{prop_characters_H}
 Extreme  traces of $\H(\infty)$ normalized by the condition $\chi(s_e)=1$ are enumerated
 by sequences $\alpha=\{\alpha_i\}$, $\beta=\{\beta_i\}$ satisfying
 $$
  \alpha_1\ge\alpha_2\ge\dots\ge 0,\quad \beta_1\ge\beta_2\ge\dots\ge 0,\quad
  \sum_i(\alpha_i+\beta_i)\le 1.
 $$
 The decomposition of the restriction of  trace $\chi^{\alpha,\beta}$ on $\H(n)$ into
  traces $\chi^\lambda$ of irreducible representations of $\H(n)$ is given by the following
 formula:
$$
 \chi^{\alpha,\beta}\rule[-2.3mm]{.4pt}{4mm}_{\H(n)}=\sum_{\lambda\in\Y_n} \Sp_{\alpha,\beta,1}
 [s_\lambda] \chi^{\lambda}
$$
\end{proposition}

Combining Propositions \ref{prop_rep_GL_H} and \ref{prop_characters_H} with Theorem
\ref{theorem_characters_of_GLB} we arrive at the following statement which should be viewed as an
infinite-dimensional analogue of Proposition \ref{prop_rep_GL_H}.

\begin{theorem}[Restriction theorem for unipotent traces]
\label{Theorem_restriction_to_Hecke}
 The restriction of the extreme unipotent trace of $\A(\GLB)$ indexed by pair of sequences
 $\alpha,\beta$ on the infinite-dimensional Hecke algebra $\H(\infty)\subset \A(\GLB)$ is the
 extreme  trace of $\H(\infty)$ indexed by the same parameters.
\end{theorem}

\section{Identification of unipotent traces with probability measures}
\label{Section_measures_unipotent}

For any $g\in\GL{n}$ let $\Cyl^\GLB_g$ denote the set of all $h\in\GLB$ such that $I^\GLB_g(h)=1$
(this function was defined at the beginning of Section \ref{section_A_as_inductive_limit}.) We call
$\Cyl^\GLB_g$ the cylindrical set of $g$. Clearly, sets $\Cyl^\GLB_g$ span the $\sigma$-algebra of
Borel sets on $\GLB$.

\begin{theorem}
\label{theorem_character_GLB_is_measure}
 Let $\chi^{\omega}$ be a unipotent  trace of $\A(\GLB)$.
 There exists a unique probability measure $\varrho^\GLB_{\omega}$ on
 $\B\subset\GLB$, such that the probabilities of cylindrical sets coincide with values of
 $\chi^{\omega}$ on the indicator functions of these cylinders. Formally,
 for any upper-triangular  matrix $g\in\GL{n}$ we have:
 \begin{equation}
 \label{eq_character_is_measure}
  \chi^{\omega}(I^\GLB_g)=\varrho^\GLB_{\omega}(\Cyl^\GLB_g)
 \end{equation}
\end{theorem}
\begin{proof}
 The $q$--Kostka numbers $K_{\mu,\lambda}^{q^{-1}}$ are polynomials in $q^{-1}$ with non-negative
 coefficients (see e.g. \cite[Section III.6]{M}). Therefore, $K_{\mu,\lambda}^{q^{-1}}\ge 0$ for
 all $\mu$ and $\lambda$. Furthermore, also $\Sp_{\alpha,\beta,1}[s_\lambda]\ge 0$ (see \cite{VK_Long},
 \cite{Kerov_book}).
 Therefore, formula \eqref{eq_characters_through_Kostka} implies that the values in the left side of \eqref{eq_character_is_measure}
 are non-negative. Consequently, we can \emph{define} the measure $\varrho^\GLB_{\omega}$ through the equation
 \eqref{eq_character_is_measure}. Definitions of the
  functions $I^\GLB_g$ and  traces readily imply that we get a well-defined probability measure.

 The uniqueness follows from Theorems \ref{theorem_character_at_simple_class} and
 \ref{Theorem_multiplicativity_of_characters} which prove that the  trace $\chi^\omega$ is
 uniquely defined by its values on unipotent classes.
\end{proof}
\noindent {\bf Remark.} Analyzing the statement of Theorem \ref{theorem_character_at_simple_class}
one sees that, for general $g$, the values of $\chi^{\omega}(I^\GLB_g)$ might be negative. Thus, if
we try to extended the measure $\varrho^\omega$ to the whole group $\GLB$, then it will no longer
be a positive measure.

\bigskip

The aim of this section is to analyze the properties of measures $\varrho^\GLB_\omega$. At this
point it is convenient to switch from $\GLB$ to $\GLU$. While all the proofs remain almost the
same, but the statements for $\GLU$ look simpler and cleaner. The reason is that unipotent upper
triangular matrices have a unique eigenvalue $1$, while general upper-triangular matrices might
have up to $q-1$ different eigenvalues and we would have to analyze the part of measure
corresponding to each of them separately.

The whole theory for $\GLU$ is very much parallel to $\GLB$. We give an overview here, the details
can be found in the Appendix. In the same way as for $\GLB$ we introduce the algebra $\A(\GLU)$ of
continuous functions on $\GLU$ with compact support taking only finitely many values. $\A(\GLU)$
is yet again an inductive limit of algebras $\A(\GLU)_n$ isomorphic to $\mathbb C(\GL{n})$,
however, the embeddings become different. While the classification of  traces of $\A(\GLU)$ is a
bit different than that of $\A(\GLB)$ there is still a class of unipotent  traces parameterized by
sequences $\alpha,\beta$. Moreover, under the identification $\A({\GLU})_n\simeq \mathbb
C(\GL{n})\simeq \A(\GLB)_n$ the restriction of unipotent  trace of $\A({\GLU})$ and $\A(\GLB)$ are
the same functions. Because of that it makes no reason to distinguish between the unipotent
traces of $\A(\GLU)$ and $\A(\GLB)$.

For $\GLU$ the group $\B$ is replaced by the subgroup $\U$ of unipotent upper-triangular matrices
and Theorem \ref{theorem_character_GLB_is_measure} transforms into Theorem
\ref{theorem_character_GLU_is_measure} (with notions of the indicator function $I^\GLU_g$ and
cylindrical set $\Cyl^\GLU_g$ analogous to those for $\GLB$) the proof of which remains the same.

\begin{theorem}
\label{theorem_character_GLU_is_measure}
 Let $\chi^{\omega}$ be a unipotent  trace of $\A(\GLU)$.
 There exists a unique probability measure $\varrho^\GLU_{\omega}$ on
 $\U$, such that for any upper-triangular  matrix $g\in\GL{n}$ we have:
 \begin{equation}
 \label{eq_character_is_measure_GLU}
  \chi^{\omega}(I^\GLU_g)=\varrho^\GLU_{\omega}(\Cyl^\GLU_g)
 \end{equation}
\end{theorem}
Note that the identification of unipotent traces of $\A({\GLU})$ and $\A({\GLB})$ implies that
random $\varrho^\GLU_{\omega}$--distributed element $\mathfrak u\in \U$ can be obtained by
conditioning random  $\varrho^{\GLB}_{\omega}$--distributed element $\mathfrak b\in\B$ to belong to
$\U$ (since for most $\omega$ we have $\varrho^{\GLB}_{\omega}(\U)=0$, to make a rigorous
definition here one should condition $\mathfrak b$ to have a unipotent $n\times n$ top--left
corner, and then send $n\to\infty$.)

The measures corresponding to unipotent  traces belong to a more general class of measures which
we now describe.

\begin{definition}
A probability measure $\varrho$ on $\U$ is called \emph{central} if $\varrho(Cyl^\GLU_g)$ depends
only on the conjugacy class, i.e.\ on the Jordan normal form of $g$.
 In other words, $\varrho$ is invariant under conjugations by elements of $\GL{\infty}$.
\end{definition}
\noindent {\bf Remark.} Conjugations, in general, do not preserve the set of upper-triangular
matrices, so the invariance means that if both $M\subset \U$ and $gMg^{-1}\subset \U$ for some
measurable set $M$ and $g\in\GL{\infty}$, then the measures of $M$ and $gMg^{-1}$ are equal.


\begin{definition}
A central measure $\varrho$ is called \emph{ergodic} if it is an \emph{extreme point} of the
convex set of all central probability measures.
\end{definition}

The following conjecture describes the set of all ergodic central measures on $\U$. For a matrix
$u\in \U$, as above, $u^{(n)}$ stays for its top-left $n\times n$ corner. Since all the
eigenvalues of $u$ are $1$s, the Jordan normal form of $u^{(n)}$ can be parameterized by a Young
diagram $\lambda$, let $\lambda_i(u,n)$ and $\lambda'_i(u,n)$ denote the row and column lengths of
$\lambda$, respectively.


\begin{conjecture}[Classification and law of large numbers for ergodic central measures]
\label{Conj_Kerov}
 Let $\vartheta$ be an ergodic central measure on $\U$. There exist two
 sequences $r_i$, $c_i$ (row frequencies and column frequencies), such that for every $i$,
 $\vartheta$--almost surely
 $$
  \lim_{n\to\infty} \frac{\lambda_i(u,n)}{n} = r_i,\quad \lim_{n\to\infty} \frac{\lambda'_i(u,n)}{n} = c_i.
 $$
 Moreover, for each pair of sequences $r=(r_1\ge r_2\ge\dots\ge 0)$ and $c=(c_1\ge c_2\ge\dots\ge 0)$
 satisfying $\sum_i(r_i+c_i)\le 1$ there exists a unique ergodic central measure $\vartheta^{r,c}$ with
 corresponding row and column frequencies.

 The $\vartheta^{r,c}$--probabilities of cylindrical sets can be expressed through the row
 frequencies and column frequencies  via the formula
\begin{equation}
 \label{eq_cylindrical_prob_for_ergodic}
  \vartheta(\Cyl^\GLU_g)=\frac{q^{-n(n-1)/2}}{(1-q^{-1})^n}
  q^{n(\lambda)} \Sp_{r,c^{(q)},1} [Q_\lambda(\cdot;q^{-1})]
 \end{equation}
 where $g\in\GL{n}$ is a unipotent upper-triangular matrix corresponding to the conjugacy class
 $\lambda$, $Q_\lambda$ is, as above, the Hall-Littlewood polynomial and for a sequence
 $c=c_1,c_2,\dots$, the sequence $c^{(q)}$ is obtained by rearranging
 two-dimensional array of numbers $(1-q^{-1})c_i q^{1-j}$, $i,j=1,2,\dots$ in decreasing order.
\end{conjecture}
\noindent {\bf Remark 1.} Conjecture \ref{Conj_Kerov} is a particular case of the conjecture on
Macdonald polynomials stated in \cite[Section II.9]{Kerov_book} and which is now known as Kerov
conjecture. A part of this conjecture was also briefly mentioned in Section 4 of \cite{Fu}.

\noindent {\bf Remark 2.} If row frequencies $r_i$ form a geometric series
$(1-q^{-1}),(1-q^{-1})q^{-1},\dots$ and column frequencies $c_i$ are zero, then (see \cite[Exercise
1 in Section III.2]{M})
$$\Sp_{r,c^{(q),1}} [Q_\lambda(\cdot;q^{-1})]=(1-q^{-1})^n q^{-n(\lambda)}$$
and $\vartheta^{r,c}$ becomes the Haar (put if otherwise, uniform) measure on $U$. The row and
column frequencies for the Haar measure on $\U$ were first found by A.~Borodin in
\cite{Borodin_GL_short},\cite{Borodin_GL}.

\noindent {\bf Remark 3.} If $r_i=0$ and $c=(1,0,0,\dots)$ then $\vartheta^{r,c}$ is the
delta-measure on the identity matrix.

\noindent {\bf Remark 4.} If $r=(1,0,\dots)$ and $c_i=0$ then $\vartheta^{r,c}$ is the uniform
measure on matrices $u\in \U$ such that $u-Id$ has maximal possible rank. In other words, all
matrix elements of $u$ on the second diagonal are non-zero.

\smallskip

Below we prove two partial results towards Conjecture \ref{Conj_Kerov}, in particular, we show that
the measure $\vartheta^{r,c}$ with cylindrical probabilities given by
\eqref{eq_cylindrical_prob_for_ergodic} is, indeed, an ergodic central measure. But, first, let us
explain the relation of measures $\vartheta^{r,c}$ to the unipotent  traces.

\begin{theorem}
 \label{Theorem_prob_measure_for_character}
 The measure $\varrho^{\GLU}_{\omega}$ is an ergodic central probability measure on $\U$.
 More precisely, if $\omega=(\alpha,\beta)$, then $\varrho^\omega=\vartheta^{\alpha^{(q)},\beta}$,
 where for a sequence
 $\alpha=\alpha_1,\alpha_2,\dots$, the sequence $\alpha^{(q)}$ is obtained by rearranging
 two-dimensional array of numbers $(1-q^{-1})\alpha_i q^{1-j}$, $i,j=1,2,\dots$ in decreasing order.
\end{theorem}
\noindent {\bf Remark. } Note the \emph{dual} role of (row and column) frequencies. On one hand,
the parameters $\alpha_i$, $\beta_i$ of unipotent characters are limit row and column frequencies
of Young diagrams parameterizing irreducible unipotent representations of $\GL{n}$, see
\cite{VK_S}, \cite{KOO}. On the other hand frequencies show up in the limit behavior of Jordan
Normal forms. The conceptual explanation of this double appearance of frequencies is unknown yet.
Somehow similar phenomena is present in the asymptotic representation theory of symmetric groups
with certain explanation given by the RSK algorithm, see \cite{VK_RSK}.

\begin{proof}[Proof of Theorem \ref{Theorem_prob_measure_for_character}]
  Theorem \ref{theorem_character_at_simple_class} implies that the cylindrical probabilities of
  measure $\varrho^{\omega}$ are given by.
 \begin{equation}
 \label{eq_x6}
  \varrho^\omega(\Cyl^\GLU_g)=q^{-n(n-1)/2}
  q^{n(\lambda)} \Sp_{\alpha,\beta,1} \left[\widetilde Q_\lambda(\cdot;q^{-1})\right]
 \end{equation}
 Note that
 $$
  \Sp_{\alpha,\beta,1} [\widetilde Q_\lambda(\cdot;q^{-1})]=(1-q^{-1})^{-|\lambda|}
  \Sp_{\alpha^{(q)},\beta^{(q),1}} [Q_\lambda(\cdot;q^{-1})].
 $$
 Comparing \eqref{eq_x6} with \eqref{eq_cylindrical_prob_for_ergodic} we conclude that
 $\varrho^\omega=\vartheta^{\alpha^{(q)},\beta}$.
\end{proof}

Now let us prove two results related to Conjecture \ref{Conj_Kerov}.

\begin{proposition}
\label{prop_list_of_ergodic} For any sequences $r=\{r_i\}$, $c=\{c_i\}$ satisfying $\sum_i
(r_i+c_i)\le 1$  the measure $\vartheta^{r,c}$ with cylindrical probabilities
\eqref{eq_cylindrical_prob_for_ergodic} is an ergodic central measure on $\U$.
\end{proposition}
\begin{proof}

The key property which we use, is the positivity of the structural constants of the multiplication
in the basis of Hall-Littlewood polynomials. In other words, in the identity
\begin{equation}
\label{eq_struct_const}
 P_\lambda(\cdot;q^{-1}) P_\mu(\cdot; q^{-1}) =\sum_\nu c_{\lambda,\mu}^\nu P_\nu(\cdot;q^{-1})
\end{equation}
When $q>1$ all the coefficients $c_{\lambda,\mu}^\nu$ are non-negative. This fact follows from the
known formulas for these coefficients (see e.g. \cite[Theorem 4.9]{R06}, \cite[Theorem 1.3]{Sc06},
\cite{KM}  and references therein). Since Hall-Littlewood $P$-polynomials and $Q$-polynomials
differ by the multiplication by an explicit constant, which is positive for $q>1$ (see
\cite[Section III.2]{M}) we can replace $P$ by $Q$ in any part of \eqref{eq_struct_const} and the
coefficients will be still positive. Moreover, \eqref{eq_struct_const} is equivalent to the
equality for skew Hall-Littlewood polynomials (see \cite[Section III.5]{M})
\begin{equation}
\label{eq_struct_for_skew}
 Q_{\nu/\mu} =\sum_\lambda c_{\lambda,\mu}^\nu Q_\lambda
\end{equation}
Again if we replace $Q$ with $P$ in either sides of \eqref{eq_struct_const} then the coefficients
 remain positive.

\medskip
Let us prove that for any sequences $r_i,c_i$, the values
$$\Sp_{r,c^{(q)},1}
[Q_\lambda(\cdot;q^{-1})]$$ are nonnegative, which will guarantee the non-negativity of
probabilities in \eqref{eq_cylindrical_prob_for_ergodic}.

First, let $r=(1,0,0,\dots)$, $c=(0,0,\dots)$. Then $\Sp_{r,c^{(q)},1} P_\lambda(\cdot; q^{-1})=0$
unless $\lambda$ is a one-row diagram, i.e. $\lambda_1=n$, $\lambda_2=0$. In the latter case
$\Sp_{r,c^{(q)},1} [P_\lambda(\cdot; q^{-1})]=1$. Thus, $\Sp_{r,c^{(q)},1} [Q_\lambda(\cdot;
q^{-1})]$ is non-negative for all $\lambda$.

\smallskip

Second, let $r=(0,0,\dots)$, $c=(1,0,0,\dots)$. The Cauchy identity for Hall-Littlewood
polynomials (see \cite[Section  III.4]{M}) yields (here $z$ and $y$ stay for two sets of
variables)
$$
 \sum_\lambda P_\lambda(z; q^{-1}) Q_\lambda(y; q^{-1}) = \exp\left(\sum_{m=1}^{\infty}
 \frac{1-q^{-m}}m p_m(z) p_m(y)\right).
$$
Applying $\Sp_{r,c^{(q)},1}$ with respect to the variables $y$ we get
\begin{multline*}
 \sum_\lambda P_\lambda(z; q^{-1}) \Sp_{r,c^{(q)},1}[Q_\lambda(y; q^{-1})] = \exp\left(\sum_{m=1}^{\infty}
 \frac{(-1)^{m-1}(1-q^{-1})^m}m p_m(z)\right)\\=\sum_{m\ge 0} (1-q^{-1})^m e_m(z)
\end{multline*}
 Since $P_\lambda=e_{|\lambda|}$ for one-column Young diagram $\lambda=(1,\dots,1)$, we
conclude that $\Sp_{r,c^{(q)},1}[Q_\lambda(y; q^{-1})]$ is zero unless $\lambda$ is a one-column
diagram. In the latter case this number is positive.

\smallskip

Third, if $r=c=(0,0,\dots)$, then $\Sp_{r,c^{(q)},1}[Q_\lambda(y; q^{-1})]$ is precisely the
coefficient of $p_1^{|\lambda|}$ in the decomposition of $Q_\lambda$ into the sum of products of
power sums $p_k$. These coefficients are known to be non-negative, see \cite[Exercise 4, Section
III.7]{M}.

\smallskip

Next, suppose that we have two specializations $\Sp_1$ and $\Sp_2$ of $\Lambda$ which map
Hall-Littlewood polynomials to non-negative numbers. Take two nonnegative numbers $a_1,a_2$ and
consider a new specialization $\Sp$, which we call \emph{mixing} of $\Sp_1$ and $\Sp_2$, given by
$$
 \Sp[p_k]=(a_1)^k \Sp_1[p_k]+(a_2)^k \Sp_2[p_k].
$$
We claim that the values of $\Sp$ on Hall-Littlewood polynomials are also nonnegative. Indeed,
this follows from the identity (see \cite[Section III.5]{M})
$$
 \Sp[Q_\lambda(\cdot; q^{-1})]=\sum_\mu (a_1)^{|\mu|}\Sp_1[Q_\mu(\cdot;q^{-1})]
 (a_2)^{|\lambda|-|\mu|} \Sp_2[Q_{\lambda/\mu}(\cdot; q^{-1})]
$$
and the positivity of the coefficients in \eqref{eq_struct_for_skew}.

Now observe that starting with three simplest specializations which we described above, one can
obtain any specialization $\Sp_{r,c^{(q)},1}$ with finitely many non-zero $r_i$s and $c_i$s
through mixing. Passing to the limit we conclude that $\Sp_{r,c^{(q)},1}[Q_\lambda]$ is
non-negative for all $\lambda$ and all sequences $r_i,c_i$ satisfying $\sum_i(r_i+c_i)\le 1$.

\bigskip

Next, let us show that the central probability measures on $U$ are in bijections with linear
functionals
$$
 \phi:\Lambda\to\mathbb C
$$
satisfying three coherency properties
\begin{enumerate}
 \item $\phi\left[Q_\lambda(\cdot;q^{-1})\right]\ge 0$ for every $\lambda$
 \item $\phi[p_1 f]=\phi[f]$ for any $f\in\Lambda$
 \item $\phi[1]=1$.
\end{enumerate}
The correspondence is pretty much given by formula \eqref{eq_cylindrical_prob_for_ergodic} and we
keep the same notations, i.e.\ given a measure $\vartheta$ the corresponding functional
$\phi_\vartheta$ is
 \begin{equation}
 \label{eq_functional_for_central_measure}
  \vartheta(\Cyl^\GLU_g)=\frac{q^{-n(n-1)/2}}{(1-q^{-1})^n}
  q^{n(\lambda)} \phi_\vartheta [Q_\lambda(\cdot;q^{-1})].
 \end{equation}
 The coherency properties 1. and 3. easily translate into the properties of a central probability measure. Let us deal with
 the coherency property 2.

 By the very definition, the cylindrical probabilities of measure $\vartheta$ should satisfy
 $$
  \vartheta(\Cyl^\GLU_g) =\sum_{h\in Ext^\GLU(g)} \vartheta(\Cyl^\GLU_h),
 $$
 where
 \begin{multline*}
 Ext^\GLU(g)=\biggl\{[h_{ij}]\in\GL{n+1}\mid\\ h^{(n)}=g\text{ and }
 h_{n+1,1}=h_{n+1,2}=\dots={h_{n+1,n}}=0, \, h_{n+1,n+1}=1\biggr\}
\end{multline*}
 is an analogue of $Ext^\GLB(g)$ of Section \ref{Section_AG_as_semisimple}.
 Let us divide $Ext^\GLU(g)$ into the groups having the same conjugacy class. We use the formula from \cite{Borodin_GL} which says that if conjugacy
 class of $g$ is given by the Young diagram $\lambda\in\Y_n$, then the number $N_{\lambda,\mu}$ of $h\in Ext^\GLU(g)$
 belonging to the conjugacy class given by the Young diagram $\mu\in Y_{n+1}$ is
 $$
  N_{\lambda,\mu}=\begin{cases} q^n q^{-\lambda'_j}(1-q^{\lambda'_j-\lambda'_{j-1}}), \text{ if }
  \mu\setminus\lambda=\square_j,\\
  0,\text{ otherwise.}
                  \end{cases}
 $$
 Here $\mu\setminus\lambda=\square_j$ means that the set-theoretical difference of the Young diagrams
 $\mu$ and $\lambda$ is a box in column $j$ and we agree that $\lambda_0=+\infty$, i.e. $q^{\lambda'_1-\lambda'_{0}}=0$.

 Therefore, the functional $\phi_\vartheta$ defined through
 \eqref{eq_functional_for_central_measure} satisfies
 \begin{equation}
 \label{eq_phi_recurrence}
  \phi_\vartheta \left[Q_\lambda(\cdot;q^{-1})\right]=\sum_{\mu\in\Y_{n+1}}
  \phi_\vartheta\left[Q_\mu(\cdot;q^{-1})\right]
  N_{\lambda,\mu} q^{n(\mu)-n(\lambda)}\frac{q^{-n}}{1-q^{-1}}
 \end{equation}
 In the same time Pierry rules for the Hall-Littlewood polynomials (see \cite[Section III.5]{M})
 yield
 $$
  (1-q^{-1})Q_\lambda(\cdot;q^{-1}) p_1 =\sum_{\mu\in Y_{n+1}: \mu\setminus\lambda=\square}
  (1-q^{\lambda'_j-\lambda'_{j-1}})Q_\mu(\cdot;q^{-1}),
 $$
 where $j$ is again the column of the box $\mu\setminus\lambda$. Therefore,
 \eqref{eq_phi_recurrence} is equivalent to
 $$
   \phi(\vartheta)\left[Q_\lambda(\cdot; q^{-1})\right]= \phi(\vartheta)\left[p_1 Q_\lambda(\cdot; q^{-1})\right].
 $$
 Since the latter equality holds for every $\lambda$ and Hall-Littlewood polynomials $Q_\lambda(\cdot;
 q^{-1})$ form a linear basis of $\Lambda$, we arrive at the coherency property 2.

 \bigskip

 Now we are in position to use the so-called \emph{Ring Theorem} (see \cite{VK_ring}, \cite{VK_Long}, \cite{Kerov_book} and also \cite[Section
8.7]{GO_Ring}). This theorem yields, that the extreme points of the convex set of
$Q_\lambda$--positive functionals on $\Lambda$ are those functional which are multiplicative
(i.e.\ are algebra homomorphism). By the very definition the functionals $\Sp_{r,c^{(q)},1}$ are
  multiplicative, they also satisfy the coherency properties. We conclude that these functionals
  are extreme and, thus, the corresponding measures $\vartheta^{r,c}$ are indeed ergodic
  central measures on $\U$.
\end{proof}

\begin{proposition}
\label{prop_ergodic_measures_as_subset}
 If $\vartheta$ is an ergodic central measure on $\U$, then there exist sequences $\alpha$, $\beta$
 satisfying \eqref{eq_summable_sequence} with $\gamma=1$ such that the measure $\vartheta$
  has the following cylindrical
 probabilities:
 $$
  \vartheta^{\alpha,\beta}(\Cyl^\GLU_g)=\frac{q^{-n(n-1)/2}}{(1-q^{-1})^n}
  q^{n(\lambda)} \Sp_{\alpha,\beta,1}[ Q_\lambda(\cdot;q^{-1})]
 $$
 where $g\in\GL{n}$ is a unipotent upper-triangular matrix corresponding to the conjugacy class
 $\lambda$.
\end{proposition}
\noindent {\bf Remark.} Thus, to prove that the measures $\vartheta^{r,c}$ exhaust the list of
ergodic central measures it remains to show that if the sequence $\beta$ is not a union of
geometric series (with denominator $q^{-1}$), then $\Sp_{\alpha,\beta,1}\left[
Q_\lambda(\cdot;q^{-1})\right]<0$ for some $\lambda$.
\begin{proof}[Proof of Proposition \ref{prop_ergodic_measures_as_subset}]
As in the proof of Proposition \ref{prop_list_of_ergodic} we identify ergodic central measures on
$\U$ with multiplicative functionals on $\Lambda$ satisfying three coherency properties. Observe
that the coefficients of the decomposition
$$
 s_\lambda(\cdot)=\sum_{\mu} c_{\lambda,\mu} Q_{\lambda}(\cdot;q^{-1})
$$
are non-negative. Indeed, up to the simple constants they coincide with $q$-Kostka numbers (see
\cite[Section III.6]{M}).

Therefore any multiplicative functional $\phi$ satisfying three coherence properties also satisfy
$$
 \phi[s_\lambda]\ge 0.
$$
But classification of the multiplicative functionals which are non-negative on Schur functions is
well-known. It is equivalent to the Thoma theorem on the characters of infinite symmetric group
$S(\infty)$, see \cite{Thoma1}, \cite{VK_S}, \cite{VK_Long}, \cite{Kerov_book} . The list of the
functionals is given by $\Sp_{\alpha,\beta,1}$ with $\alpha,\beta$ satisfying
\eqref{eq_summable_sequence} with $\gamma=1$.
\end{proof}

\section{Grouppoid construction for the representations of $\GLB$}
\label{Section_reps}

In this section we give an explicit construction for the representations of $\GLB$ corresponding
to a large class of the extreme unipotent traces of $\A(\GLB)$.

\subsection{Generalities}

Generally speaking, we are going to construct irreducible \emph{generalized spherical
representations} of pair $(\GLB\times\GLB,\GLB)$. Let us introduce some definitions first.

The well-known principle (see e.g.\ \cite[Section 13]{Dixmier}) identifies unitary representations
of a locally-compact group $G$ with $*$--representations of $L_1(G)$ (with respect to Haar
measure) and we will silently use this identification where it leads to no confusions.

\begin{definition}
A generalized spherical representation of $(\GLB\times\GLB,\GLB)$ is a triplet:
\begin{enumerate}
\item Unitary (continuous) representation $\pi$ in a Hilbert space $H$ with scalar product $\langle \cdot,\cdot \rangle$:
$$
 \pi:\GLB\times\GLB \to U(H),
$$
\item Dense subspace $H_1\subset H$ equipped with a norm $|\cdot|$,
\item Linear functional (distribution) $v\in H_1'$
\end{enumerate}
Satisfying the following conditions:
\begin{enumerate}
\item The inclusion $i:(H_1,|\cdot|)\to (H,\langle \cdot,\cdot \rangle)$ is continuous,
\item For any $(g,h) \in \GLB\times \GLB$, $\pi(g,h) H_1\subset H_1$
\item For any $a\in\A({\GLB})$ we have $\pi(a,e)v\in H_1$ and $\pi(e,a)v\in H_1$. In other words, there
exists $w(a)$ such that $\langle v, \pi(a,e) x\rangle=\langle w(a),x \rangle$ for every $x\in
H_1$, and similarly for $(e,a)$.
\item The span of $\{\pi(a,b) v\mid a,b\in \A(\GLB)\}$ is  dense in $H$
\item For any $g\in\GLB$, we have $\pi(g,g) v= v$.
\item $(\pi(I_\B) v,v)=1$.
\end{enumerate}
\end{definition}

This definition is just a generalization of the well-known definition of a spherical
representations of the Gelfand pair. The theory of spherical representations for the infinite
symmetric group $S(\infty)$ and infinite-dimensional unitary group $U(\infty)$ was developed by
G.Olshanski and his collaborators, see \cite{Olsh_Howe_short}, \cite{Olsh_Howe}, \cite{Olsh_S}.
The novelty in the present paper is the fact that the distinguished vector $v$ no longer belongs
to the Hilbert space, but becomes a distribution.

\begin{definition}
A spherical function of  a generalized spherical representation $(\pi, H_1, v)$  is defined as
$$
  \chi(a,b)=\langle\pi(a,b)v,v\rangle,
$$
where one of the variables $a$ or $b$ should belong to $\A(\GLB)$ (and other might be either an
element of $\A(\GLB)$ or, more generally, an element of $\GLB$).
\end{definition}
Clearly, the restriction of the spherical function to its first coordinate (i.e.\ to pairs $(a,e)$
) gives a trace of $\A(\GLB)$.

The converse is also true, i.e.\ given a trace of $\A(\GLB)$ we can, in principle, construct the
corresponding representation (using a version of the Gelfand-Naimark-Segal construction). However,
the general construction is quite abstract and we seek for an explicit description of the
representations corresponding to the unipotent  traces $\chi^\omega$.

We could avoid the notion of a general spherical representation and use von Neuman factors
instead, however, our approach seems to show more hidden structure. The following simple
proposition explains how to pass to the factor-representations.

\begin{proposition}
 Let $(\pi,H_1,v)$ be a generalized spherical representation such that the corresponding traces
  $\chi$ of $\A(\GLB)$ is extreme. The restriction of $\pi$ on the first coordinate is von Neumann semifinite (i.e.\
 either type $I$ or type $II$) factor representation of group $\GLB$ in the cyclic span of $v$.
\end{proposition}
\begin{proof}
 Let $\mathcal V\subset \mathcal B(
 H)$ denote the minimal von Neumann algebra containing all operators $\pi(g,e)$,
 $g\in\GLB$. Let $\chi'$ denote the (unique) extension of the  trace $\chi$ of $\A(\GLB)$
 on $\mathcal V$. Clearly, $\chi'$ is a semifinite trace of
 $\mathcal V$. Note that $\chi'$ is extreme. Indeed, if $\chi'=\chi'_1+\chi'_2$, then $\chi$ has a similar decomposition
  and we get a contradiction with extremality of $\chi$. Extremality of $\chi'$ implies that $\mathcal V$ is a
 von Neumann factor.
\end{proof}

\noindent {\bf Remark.} As we will see below, both type $I$ and type $II$ factor representations
arise.

\subsection{Two simplest type $I$ examples}

\label{subsect_single_param_reps}

Recall that unipotent representations are parameterized by two sequences
$\alpha_1\ge\alpha_2\ge\dots\ge$ and $\beta_1\ge\beta_2\ge\dots\ge 0$ such that
$\sum_i(\alpha_i+\beta_i)\le 1$. We start from considering some simplest cases.

First, suppose that $\alpha_1=1$ with all other parameters being zeros. In this case the desired
representation is just the identity representation. I.e.\ $H$ is 1-dimensional vector space,
$H_1=H$, $\pi$ maps all elements of $\GLB\times\GLB$ to identity operator and $v$ is a unit vector
in $H$.

Next, let $\beta_1=1$, and let all other parameters be zeros.

Let $St_n$ be the \emph{Steinberg} representation of $\GL{n}$ (see \cite{Steinberg}, \cite{Hum}).
This representation can be realized as the left representation of $\GL{n}$ in the right ideal of
$\mathbb C(\GL{n})$ spanned by the element
$$
  s=\sum_{\sigma\in \Sym_n} (-1)^{\sigma} e_\sigma \sum_{g\in \B_n} e_g
$$
It is well known that a linear basis of $H(St_n)$ can be chosen to be
$$\{e_g s\mid g\in \U_n\},$$
where $\U_n$ is a subgroup of unipotent upper triangular matrices in $\GL{n}$. The dimension of
$H(St_n)$ is $q^{n(n-1)/2}$.

The representation $\St_{n-1}$ of $\A(\GLB)_{n-1}$ is naturally included into the representation
$\St_n$ of $\A(\GLB)_n$ as the subspace of $U_n^n$-invariant vectors (see Theorem
\ref{Theorem_parabolic_restriction} for the definition of the group $U_n^n$). Let $\St_\infty^0$
denote the inductive limit of the representations $\St_n$ with respect to the above embeddings.
Note that each $H(\St_n)$ has a (unique up to a multiplication by a constant) $\GL{n}$--invariant
scalar product and these scalar products can be choosen to agree with the above embeddings. Thus,
the space $H(\St^0_\infty)$ is equipped with a scalar product. Let $H(\St_\infty)$ denote
$*$-representation of $\A$ in the completion of the space $H(\St^0_\infty)$.

Now let $\mathcal H$ denote the Hilbert space $H(\St_\infty)^*\otimes H(\St_\infty)$ of
Hilbert-Schmidt operators in $H(\St_\infty)$. We have a natural $*$-representation $\pi$ of
$\A\times \A$ in $\mathcal H$. It can be extended to a non-degenerate representation of
$L_1(\GLB)\otimes L_1(\GLB)$ and, thus, to a unitary representation of $\GLB\times\GLB$, which we
denote by the same letter $\pi$. Let $\mathcal H_1\subset \mathcal H$ be the subspace of
trace-class operators and let the functional $v\in\mathcal H_1'$ be trace. (If we identify $H_1'$
with the space $\mathcal B(\mathcal H)$ of bounded linear operators, then $v$ corresponds to the
identity operator.) Note that if $a\in\A_n$, then the image of the operator $\St_\infty(a)$ lies
in $H(\St_n)$. Therefore, $\St_\infty(a)$ has finite rank. It follows that $\pi((a,b))v\in
\mathcal H_1$ for any $(a,b)\in\A\times \A$. All other properties are trivial and we conclude that
$(\mathcal H, \mathcal H_1, v)$ is a generalized spherical representation. One immediately checks
that the spherical function of this representation corresponds to the trace with $\beta_1=1$ and
all other parameters being zeros.

\subsection{Representations related to grassmanian}
\label{Subsect_grassman}

We next proceed to the construction of the representation with $\alpha_1=t_1$, $\alpha_2=t_2$,
$t_1+t_2=1$. Our construction has lots of similarities with \emph{grouppoid construction} of
\cite{VK_S} for the realization of factor representation of the infinite symmetric group
$S(\infty)$.

Let $V$ be the infinite-dimensional linear space over $\F$ with basis $e_1,e_2,\dots$. Denote
$V_i=\langle e_1,\dots,e_i \rangle$. In what follows we use an infinite-dimensional analogue of
the well-known decomposition of grassmanian into Schubert cells.

\begin{definition}
For a subspace $X$ of $V$ with $d_i=\dim(X\bigcap V_i)$, the \emph{symbol} of $X$ is the $0-1$
sequence $d_{i}-d_{i-1}$, where we agree that $d_0=0$:
$$
 Sym(X):=(d_1-d_0, d_2-d_1,\dots).
$$
\end{definition}

\begin{definition} For a $0-1$ sequence $x$ let \emph{Schubert cell of} $x$ denote the set of all subspaces
with symbol $x$:
$$
 Sch(x)=\{X\subset V\mid Sym(X)=x\}.
$$
\end{definition}
We fix a distinguished \emph{coordinate} subspace in $Sch(x)$ which is
$$
C(x)=\langle e_i \mid x_i=1\rangle.
$$

In the same way if $X$ is a subspace of $V_n$, then its $n$-dimensional symbol $Sym^n(X)$ is the
$0-1$ sequence $(d_1-d_0,\dots,d_n-d_{n-1})$ of length $n$. For a $0-1$ sequence $(x_1,\dots,x_n)$
we define a finite Schubert cell
$$
Sch^n(x)=\{X\subset V_n\mid Sym^n(X)=x\}.
$$
By a simple linear algebra we have
\begin{equation} |Sch^n(x)|=q^{\sum_{i=1}^n (i x_i) -
m(m+1)/2}, \label{eq_Schubert_size}
\end{equation}
where $m=\sum_{i=1}^n x_i$.

\noindent {\bf Remark.} Another way to rewrite \eqref{eq_Schubert_size} is
$$
|Sch^n(x)|=q^{{\rm inv}(-x)},
$$
where ${\rm inv}(-x)$ is the number of \emph{inversions} in $-x$. In other words, it is the number
of pairs $i<j$ such that $x_i<x_j$. Similar formula still holds when we pass from grassmanian to
more complicated flag varieties. This makes a link to $q$--exchangeability and
$\GL{\infty}$--invariant measures on flags of \cite{GO_q}.

\smallskip

Let $\nu_x^n$ denote the uniform probability measure on the finite set $Sch^n(x)$.  Thus, for a
subspace $X$ in $V_n$ we have
$$
 \nu_x^n(X)=\begin{cases}
     q^{m(m+1)/2 -\sum_{i=1}^n (i x_i)}, \quad X\in Sch^n(x),\\
     0,\text{ otherwise.}
   \end{cases}
$$

For a space $W$ ($W$ will be either $V$ of $V_n$) let $Gr(W)$ be the set of all subspaces of $W$.
Note that $\GLB$ naturally acts in $Gr(V)$. We equip $Gr(V)$ with a topology of $\GLB$--space
(i.e.\ elementary open neighborhood of a point $x$ is the image of the action on $x$ of an open
neighborhood of identity element in $\GLB$) and corresponding $\sigma$--algebra of Borel sets.
$Gr(V)$ is a union of Schubert cells, every cell is a measurable subset of $Gr(V)$ and is a
$\B$-orbit. Let $\pi_x$ be the map:
$$
 \pi_x: \B \to Gr(V),\quad \pi(g)=g C(x).
$$
Let measure $\nu_x$ be the image of the Haar measure on $\B$ with respect to $\pi_x$. By its
definition $\nu_x$ is a unique $\B$-invariant probability measure supported on $Sch(x)$.

Let $\pi^{(n)}$ be the projection
$$
 \Pi^{(n)}: Gr(V)\to Gr(V_n),\quad \pi^{(n)}(X)=X\bigcap V_n,
$$
then the image of $\nu_x$ with respect to the map $\pi^{(n)}$ is precisely the uniform probability
measure $\nu_{(x_1,\dots,x_n)}^n$ on the finite Schubert cell $Sch^n((x_1,\dots,x_n))\subset
Gr(V_n)$.

Let us introduce an important probability measure $\eta_{t_1,t_2}$ on $Gr(V)$. Let $\phi$ denote
the map
$$
 \phi: \{0,1\}^{\infty}\times \B \to Gr(V),\quad (x,g)\to g C(x).
$$
\begin{definition}
The measure $\eta_{t_1,t_2}$ is the $\phi$-pushforward of the product of Bernoulli measure with
probability of $1$ being $t_1$, and Haar measure $\mu_\B$ on $\B$. In other words, to get a random
element of $Gr(V)$ distributed according to the measure $\eta_{t_1,t_2}$ we, first, sample a $0-1$
sequence $x$ from the Bernoulli measure and then take an element of $Sch(x)$ distributed according
to $\nu_x$.
\end{definition}
We also let $\eta_{t_1,t_2}^n$ be the $\Pi^{(n)}$ pushforward of $\eta_{t_1,t_2}$. Our definitions
imply that for $X\in Gr(V_n)$ with symbol $(x_1,\dots,x_n)$ we have
\begin{equation}
\label{eq_quasiinv_eta_finite}
 \eta_{t_1,t_2}^n (X)=q^{m(m+1)/2 -\sum_{i=1}^n (i x_i)} t_1^{\sum_i x_i} t_2^{n-\sum_i x_i}.
\end{equation}

The following two propositions explain the relation between $\eta_{t_1,t_2}$ and action of $\GLB$.

\begin{proposition}[Fundamental cocycle of the action on grassmanian]
\label{prop_eta_quasiinvar} The measure $\eta_{t_1,t_2}$ is quasi-invariant with respect to the
action of $\GLB$. The cocycle of the action of $\GLB$ is given by
$$
 \frac{\eta_{t_1,t_2}(g \cdot dX)}{\eta_{t_1,t_2}(dX)}=q^{\sum_k
 k(Sym(X)_k-Sym(gX)_k)}.
$$
\end{proposition}
\begin{proof}
 By the definition $\eta_{t_1,t_2}$ is $\B$--invariant. Thus, it remains to consider $g\in\GL{n}$
 for arbitrary $n$. But then the computation of the cocycle of $\eta_{t_1,t_2}$ boils down to the
 computation for $\eta^{n}_{t_1,t_2}$ which is straightforward from \eqref{eq_quasiinv_eta_finite}.
\end{proof}

\begin{proposition}
\label{Prop_no_equiv}
 If $t_1$ and $t_2$ are nonzero, then there is no finite or $\sigma$--finite $\GLB$--invariant measure on $Gr(V)$ equivalent (i.e.\
 with the same sets of measure zero) to $\eta_{t_1,t_2}$.
\end{proposition}
\noindent {\bf Remark. }  The classification of all \emph{finite} $\GLB$--invariant measures on
$Gr(V)$ was recently found in \cite{GO_qPascal}. But the theorem of \cite{GO_qPascal} is not enough
for us, since we also want to deal with $\sigma$--finite measures.
\begin{proof}[Proof of Proposition \ref{Prop_no_equiv}]
 We argue by contradiction. Suppose that such measure $\varphi$ exists. Then by the
 Radon--Nickodim theorem, $\varphi$ should have a positive density $f(x)$ with respect to
 $\eta_{t_1,t_2}$. Choose two positive numbers $a<b$ such that $b/a<q$ and the set
 $$
  M=\{x\in Gr(V)\mid a<f(x)<b\}
 $$
 has a positive measure $\mathfrak m$ with respect to $\eta_{t_1,t_2}$.

 Fix a large enough integer $n$. For a subspace $Y\in Gr(V_n)$ denote
 $$
  U(Y;n)=\{X\in Gr(V)\mid X\bigcap V_n=Y\}.
 $$
 Observe that $\GL{n}$--orbits in $Gr(V_n)$ are parameterized by integers $r=0,1,\dots,n$ which represent the sum of the coordinates of symbols of subspaces in the orbit and
 let $\mathcal O^r$ be the corresponding orbit:
 $$\mathcal O^r=\left\{X\in Gr(V_n) \mid \sum_{k=1}^n Sym(X)_k=r\right\}.$$
 Further, set
 $$
  \mathcal O^r_m=\left\{X\in \mathcal O^r\mid \sum_{k=1}^n k Sym(X)_k=m\right\}.
 $$

 Now choose $r$ such that
\begin{equation}
\label{eq_x13}
 \frac{ \eta_{t_1,t_2}\left(\bigcup_{X\in\mathcal O^r} U(X;N) \bigcap M\right)} {
\eta_{t_1,t_2}\left(\bigcup_{X\in\mathcal O^r} U(X;N)\right)} \ge \mathfrak m/2.
\end{equation}
Since $t_1,t_2>0$, the definition of $\eta_{t_1,t_2}$  implies that for large enough $n$ we can
assume $0<r<n$. Let $N_r(m)$ be the number of $0-1$ sequences $\{x_i\}$ of length $n$ such that
$\sum_k x_k=r $ and $\sum_k k x_k=m$. Then we have
\begin{equation}
\label{eq_x16} \frac{N_r(m)}{\sum_l N_r(l)} =\frac{ \eta_{t_1,t_2}\left(\bigcup_{X\in \mathcal
O^r_m} U(X;N)\right)}{ \eta_{t_1,t_2}\left(\bigcup_{X\in\mathcal O^r} U(X;N)\right)},
\end{equation}
\begin{equation}
\label{eq_x15} \sup_{m} \frac{N_r(m)}{\sum_l N_r(l)} \le H(r,n),
\end{equation}
and the function $H(r,n)$ tends to $0$ as $n\to\infty$ uniformly in $0<r<n$.

Now let
$$
 c_m= \frac{
\eta_{t_1,t_2}\left(\bigcup_{X\in\mathcal O^r_m} U(X;N) \bigcap M\right)} {
\eta_{t_1,t_2}\left(\bigcup_{X\in\mathcal O^r} U(X;N)\right)}.
$$
Inequality \eqref{eq_x13} means that $\sum_m c_m \ge \mathfrak m/2$ and
\eqref{eq_x16},\eqref{eq_x15} imply that $\sup_m c_m \le H(r,n)$.


Fix arbitrary $Z\in\mathcal O^r$ and for each other $X\in\mathcal O^r$ choose $g(X)\in\GL{n}$
sending $Z$ to $X$. In particular, set $g(Z)$ to be the unit element of $\GL{n}$.
 For two subspaces $X,Y\in\mathcal O^r$ define
 $$
  M_X(Y)=g(Y) g(X)^{-1} \bigl( M\cap U(X;n) \bigr) \subset U(Y;n).
 $$
The definitions imply that if $X,Y\in\mathcal O^r_m$, then for any $U\in\mathcal O^r$ we have
$M_X(U)=M_Y(U)$ (here and below all such identities should be understood up to the sets of
$\eta_{t_1,t_2}$ measure zero). Indeed, if $U=Y$, then using Proposition \ref{prop_eta_quasiinvar}
we see that under the action of $g(Y) g(X)^{-1}$ restricted to $U_{X;n}$, both measures $\phi$ and
$\eta_{t_1,t_2}$ are invariant and, thus, the density $f$ also does not change and
$M_X(Y)=U(Y;n)\cap M=M_Y(Y)$. Applying $g(U) g(Y^{-1})$ we get the general $U$ case. Similarly, if
$X\in\mathcal O^r_m$ and $Y\in\mathcal O^r_{m'}$ with $m\ne m'$, then for any $U\in\mathcal O^r$,
the sets $M_X(U)$ and $M_Y(U)$ are disjoint --- this immediately follows from the observation that
the inequality $a<f(x)<b$ breaks down when we multiply $f(x)$ by any integral power of $q$.

Further, observe that for any $X\in\mathcal O^r$  the sum over a cell
\begin{equation}
\label{eq_x14}
 \sum_{Y\in Sch^n(y)}
\eta_{t_1,t_2}(M_X(Y))
\end{equation}
does not depend on the choice of cell $Sch^n(y)$. Indeed, by Proposition
\ref{prop_eta_quasiinvar} $\eta_{t_1,t_2}(M_X(Y))$ differs from
$\eta_{t_1,t_2}(M_X(X))$ by a power of $q$ and the same power appears when we compare using \eqref{eq_Schubert_size}
the number of terms in \eqref{eq_x14} for different cells.

Now fix a cell $Sch^n(y)$. For each $m$ we can choose $X\in \mathcal O^r_m$ and form the set
$\bigcup\limits_{Y\in Sch^n(y)} M_X(Y) $ of $\eta_{t_1,t_2}$--measure
$$
 \frac{c_m}{N_r(m)} \cdot \eta_{t_1,t_2}\left(\bigcup_{X\in\mathcal O^r} U(X;N)\right).  $$
 Note that all these sets are
disjoint. Therefore, summing over all $m$ and all Schubert cells in $\mathcal O^r$ we conclude that
\begin{equation}
\label{eq_x17}
 \sum_{m,l} \frac{c_m}{N_r(m)} N_r(l) \le 1.
\end{equation}
On the other hand,
$$
 \sum_{m,l} \frac{c_m}{N_r(m)} N_r(l)= \left(\sum_m \frac{c_m}{N_r(m)} \right)\left(\sum_l N_r(l)
 \right) \ge \left(\sum_m \sqrt{c_m}\right)^2 \ge \frac{(\sum_m c_m)^2}{\sup_m c_m } \ge \frac{\mathfrak
 m^2}{4 H(r,n)},
$$
which for large $n$ contradicts \eqref{eq_x17}.
\end{proof}

In order to get the desired spherical representation we need to introduce a more complicated
space. The construction of this space has similarities with analogous construction for infinite
symmetric group $S(\infty)$, see \cite{VK_S}, \cite{VT} with some ideas tracing back to the papers
of F.~J.~Murray and J.~von~Neumann \cite{MN}, \cite{N}.

Let
$$
Gr^2(V)=\{(X,Y)\in Gr(V)\times Gr(V)\mid X=gY, \text{ for some } g\in \GLB\}.
$$
We equip the set $Gr^2(V)$ with a topology, the elementary open neighborhoods of a point $(X,Y)$
are indexed by numbers $n=0,1,2,\dots$ and
\begin{multline*}
 U^n(X,Y)=\Bigl\{(Z,W)\in Gr^2(V)\mid Z\bigcap V_n=X\bigcap V_n,\ W\bigcap V_n= Y\bigcap V_n, \\
 \dim(Z\bigcap V_k)=\dim(W\bigcap V_k),\text{ for }k\ge n\Bigr\}.
\end{multline*}
Note that $U^n(X,Y)$ actually depends only on $X\cap V_n$, $Y\cap V_n$ and this set is empty
unless $\dim(X\cap V_n)=\dim(Y\cap V_n)$. In this topology $Gr^2(V)$ is locally compact. The group
$\GLB\times\GLB$ naturally acts in $Gr^2(V)$ and the action is continuous in the introduced
topology.

 Now we introduce a measure
$\rho_{t_1,t_2}$  on $Gr^2(V)$ which is quasiinvariant with respect to the action of
$\GLB\times\GLB$.

Let $x,y$ be two infinite $0-1$ sequences. We write $x\sim y$ if there exists $N$ such that
$x_n=y_n$ for $n>N$ and $\sum_{n=1}^N x_n =\sum_{n=1}^N y_n$. Note that $(X_1,X_2)\in Gr(V)\times
Gr(V)$ belongs to $Gr^2(V)$ is and only if $Sym(X_1)\sim Sym(X_2)$.

Denote $\mathcal T$ the set of pairs $(x,y)$ of infinite $0-1$ sequences such that $x\sim y$.
$\mathcal T$ is equipped with sigma algebra spanned by the sets $A_{i_1,\dots,i_n}^{j_1,\dots,j_n}$,
where $i_1,\dots, i_n$ and $j_1,\dots, j_n$ are two $0-1$ sequences such that $\sum i_k=\sum j_k$
$$
 A_{i_1,\dots,i_n}^{j_1,\dots,j_n}=\{(x,y)\in \mathcal T\mid x_1=i_1,\dots, x_n=i_n,\, y_1=j_1,\dots,
 y_n= j_n,\, y_k=x_k, \text{ for } k>n\}.
$$
We define the measure $R_{t_1,t_2}$  by
$$
 R_{t_1,t_2}(A_{i_1,\dots,i_n}^{j_1,\dots,j_n})=t_1^{\sum_k i_k} t_2^{n-\sum_k i_k}.
$$

Let $\psi$ be the map
$$
 \psi:\mathcal T \times \B \times \B \to Gr^2(V),\quad \psi((x,y,g,h))= (gC(x),h C(y),
$$
and let $\rho_{t_1,t_2}$  be the push-forward of the measure $R_{t_1,t_2}\otimes \mu_\B\otimes
\mu_\B$ with respect to $\psi$.

The following proposition gives a more direct description of the measure $\rho_{t_1,t_2}$.

\begin{proposition}
\label{proposition_measure_evaluation} Let $n=0,1,2,\dots$ and $(X,Y)\in Gr^2(V)$. Suppose that
$\dim(X\cap V_k)=\dim(Y\cap V_k)$ for $k\ge n$ and denote
$$m:=\dim(X\cap V_n)=\dim(Y\cap V_n)=\sum_{k=1}^n Sym(X)_k.$$ We have
$$
 \rho_{t_1,t_2}(U^n(X,Y))=q^{m(m+1)-\sum_k k(Sym(X)_k+Sym(Y)_k)} t_1^m t_2^{n-m}.
$$

\end{proposition}
\begin{proof}
 Observe that if $(X',Y')$ are such that $Sym(X')=Sym(X)$ and $Sym(Y')=Sym(Y)$, then there exist
 $(g,h)\in \B\times\B$ such that $X'=gX$, $Y'=hY$. Therefore
 $$
  U^n(X',Y')=(g,h)U^n(X,Y),
 $$
 hence, by the definition of the measure,
 $$
 \rho_{t_1,t_2}(U^n(X,Y))=\rho_{t_1,t_2}(U^n(X',Y')).
 $$
 Now note that
 \begin{equation}
 \label{eq_x1}
  \psi(A_{Sym(X)_1,\dots,Sym(X)_n}^{Sym(Y)_1,\dots,Sym(Y)_n}\times \B\times \B)=\bigcup_{(Z_i,W_i)} U^n(Z,W),
 \end{equation}
 where $Z_i$ goes over $q^{\sum_k k(Sym(X)_k-m(m+1)/2}$ subspaces of $V$ such that $Sym(Z_i)=Sym(X)$
 and $Z_i\bigcap V_n$ are pairwise distinct; $W$ goes over $q^{\sum_k k(Sym(Y)_k-m(m+1)/2}$ subspaces of $V$ such that $Sym(W_i)=Sym(Y)$
 and $W_i\bigcap V_n$ are pairwise distinct. Evaluating $\rho_{t_1,t_2}$ of both sides of \eqref{eq_x1} we get the desired formulas.
\end{proof}

\begin{corollary}
\label{corol_rho_quasiinvar} The measure $\rho_{t_1,t_2}$  is quasi-invariant with respect to the
action of $\GLB\times\GLB$. The corresponding cocycle is given by
$$
 \frac{\rho_{t_1,t_2}((g,h) \cdot d(X,Y))}{\rho_{t_1,t_2}(d(X,Y))}=q^{\sum_k
 k(Sym(X)_k-Sym(gX)_k) + \sum_k k(Sym(Y)_k-Sym(hY)_k)}.
$$

\end{corollary}
\begin{proof}
 This follows from Proposition \ref{proposition_measure_evaluation} and the fact that if
 $(g,h)\in\GLB\times\GLB$ and $n$ is large enough integer, then
 $$
  (g,h)U^n(X,Y)=U^n(gX,hY).
 $$
\end{proof}
\noindent {\bf Remark.} It is easy to replace the measure $\rho_{t_1,t_2}$ with an equivalent one
$\widehat \rho_{t_1,t_2}$ which would be invariant with respect to the action of the subgroup
$\GLB\times\{e\}\subset \GLB\times\GLB$. However, it is not possible to achieve the invariance with
respect to the whole group $\GLB\times\GLB$.

\smallskip
Further, let $\pi_{t_1,t_2}$ denote the usual unitary representation of $\GLB\times\GLB$ in the
$L_2(Gr^2(V),\rho_{t_1,t_2})$. In other words, for $f\in L_2(Gr^2(V),\rho_{t_1,t_2})$ and
$(g,h)\in\GLB\times\GLB$ we have
$$
 \left[\pi_{t_1,t_2}(g,h)f \right] (X,Y)= f(g^{-1}X, h^{-1}Y) \sqrt{\frac{\rho_{t_1,t_2}((g^{-1},h^{-1}) \cdot
 d(X,Y))}{\rho_{t_1,t_2}(d(X,Y))}}.
$$

Let $C^0(Gr^2(V))$ denote the space of continuous functions on $Gr^2(V)$ with compact support
equipped with supremum-norm. We have natural inclusions
$$
 C^0(Gr^2(V))\subset L_2(Gr^2(V),\rho_{t_1,t_2}) \subset \big(C^0(Gr^2(V))\big)^*.
$$
 Consider the unitary representation of $\GLB\times\GLB$
dual to the restriction of $\pi_{t_1,t_2}$ on  $C^0(Gr^2(V))$. Somewhat abusing the notations we
will use the same symbol $\pi_{t_1,t_2}$ for this representation.

Let $v_{t_1,t_2}\in C^0(Gr^2(V))^*$ denote the linear functional:
$$
 v_{t_1,t_2}:C^0(Gr^2(V))\to\mathbb C,\, v_{t_1,t_2}(f)=\int_{Gr(V)} f(X,X) \eta_{t_1,t_2}(dX).
$$
Further, let $Sp(v_{t_1,t_2})$ denote the $\GLB\times\GLB$ cyclic span of the linear functional
$v_{t_1,t_2}$ and let $\hat\pi_{t_1,t_2}$ denote the restriction of $\pi_{t_1,t_2}$ on
$L_2$--closure of $Sp(v_{t_1,t_2}) \cap C^0(Gr^2(V))$.

\begin{theorem}
\label{Theorem_construction_gras}
 The triplet $(\hat \pi_{t_1,t_2}, C^0(Gr^2(V)) \cap Sp(v_{t_1,t_2}), v_{t_1,t_2})$ is a generalized spherical
 representation. Its spherical function gives the extreme unipotent  trace of $\A(\GLB)$ with
 parameters $\alpha_1=t_1$, $\alpha_2=t_2$. The restriction of this representations on the first
 component is von Neumann factor representation of type $II_\infty$.
\end{theorem}
\begin{proof}
 Let us check the 6 properties of a generalized spherical representation.
\begin{enumerate}
\item
 The natural map $i: C^0(Gr^2(V))\to L_2(Gr^2(V),\rho_{t_1,t_2})$ is, indeed, a continuous inclusion. This
 follows from our choice of topology.
\item
 Since for any element $(g,h)\in\GLB\times\GLB$ and any elementary open neighborhood $U^n(X,Y)$ (with large enough
 $n$) we have $(g,h) U^n(X,Y)=U^n(gX,gY)$, thus, the action of $\GLB\times\GLB$ maps open
 sets to open sets and, therefore, preserves the space of continuous functions with compact support.
\item
 It suffices to prove that $(I_{e(n)},e)v_{t_1,t_2}$ is a continuous function with compact support.
 Let $\BI_n$ denote the subgroup $\Cyl^\GLB_{e(n)}\subset \B$. By the definition
 \begin{multline*}
 ((I_{e(n)},e)v_{t_1,t_2},f)=\int_{X\in Gr(V)} \int_{g\in \BI_n} ((g,e)\cdot f)(X,X) \mu(dg) \eta_{t_1,t_2}(dX)\\=
 \int_{X\in Gr(V)} \int_{g\in \BI_n} \sqrt{\frac{(g^{-1},e)\rho_{t_1,t_2}}{\rho_{t_1,t_2}}(X,X)} f(g^{-1}X,X) \mu(dg) \eta_{t_1,t_2}(dX)
\\= \int_{X\in Gr(V)} \int_{g\in \BI_n} f(g^{-1}X,X) \mu(dg) \eta_{t_1,t_2}(dX)
\end{multline*}
For $X,Y\in Gr(V)$ let $\mathcal U^k(X,Y)$ denote the indicator function of the set $U^k(X,Y)$. To
analyze the last integral we set $f=\mathcal U^k(I,J)$, $k>n$, $I,J\in V^k$ and compute $
\int_{g\in \BI_n} f(g^{-1}X,X)\mu(dg) $. Observe that if $X\bigcap V_k\ne J$ or $X\bigcap V_n\ne
I\bigcap V_n$ or $\dim(X\bigcap V_\ell)\ne \dim (I\bigcap V_\ell)$ for some $n\le \ell \le k$,
then $(g^{-1}X,X)$ does not belong to $U^k(I,J)$ and the integral vanishes. Otherwise, it is equal
to
$$
 \int_{g\in \BI_n} f(g^{-1}X,X)\mu(dg)= \mu(\BI_n)/M,
$$
where $M$ is the number of $Z\in Gr(V_k)$ such that $Z\bigcap V_n= I\bigcap V_n$ and
$\dim(Z\bigcap V_\ell)=\dim(I\bigcap V_\ell)$ for all $n\le \ell \le k$. We have
$$
1/M=q^{\dim(I)(\dim(I)+1)/2-\dim(I\cap V_n)(\dim(I\cap
 V_n)+1)/2-\sum_{\ell=n+1}^k \ell Sym(I)_\ell}.
$$
Therefore, the double integral equals
\begin{multline}
\label{eq_x3} \int_{X\in Gr(V)} \int_{g\in \BI_n} f(g^{-1}X,X) \mu(dg)
\eta_{t_1,t_2}(dX)\\=\mu(\BI_n) q^{\dim(I)(\dim(I)+1)/2-\dim(I\cap V_n)(\dim(I\cap
 V_n)+1)/2-\sum_{\ell=n+1}^k \ell Sym(I)_\ell} \eta_{t_1,t_2}^k(J)
\end{multline}
if $I\bigcap V_n=J\bigcap V_n$ and $\dim(I\bigcap V_\ell)=\dim(J\bigcap V_\ell)$ for $n\le\ell\le
k$, otherwise the double integral vanishes. \eqref{eq_x3} together with formulas for
$\rho_{t_1,t_2}(U^k(I,J))$ and $\eta_{t_1,t_2}^k(J)$ imply
\begin{multline}
\label{eq_functional_computation} \int_{X\in Gr(V)} \int_{g\in \BI_n} f(g^{-1}X,X) \mu(dg)
\eta_{t_1,t_2}(dX)\\= \int_{Gr^2(V)} v^n(X,Y) f(X,Y) \rho_{t_1,t_2}(d(X,Y)),
\end{multline}
 where
$$
 v^n_{t_1,t_2}(X,Y) =\begin{cases} \mu(\BI_n) q^{-\dim(X\cap V_n) (\dim(X\cap V_n)+1)/2 +\sum_{\ell=1}^n
 \ell Sym(X)_\ell},\text{ if } (X,Y)\in L_n,\\ 0,\text{ otherwise}, \end{cases}
$$
\begin{multline*}
  L_n=\{(X,Y)\in Gr^2(V)\mid \\ X\bigcap V_n=Y\bigcap V_n,\, \dim( X\bigcap V_k)=\dim( Y\bigcap V_k) \text{ for
  }k>n\}.
\end{multline*}
 Since the linear span of the functions $\mathcal U^k(I,J)$ (with various $k$) is dense in
 $L_2(Gr^2(V),\rho_{t_1,t_2})$, the equality \eqref{eq_functional_computation} holds for a general
 function $f(X,Y)$.  Thus, $(I_{e(n)},e)v_{t_1,t_2}=v^n_{t_1,t_2}$ is, as desired, a continuous function with compact support.
\item Since we work in the span of $v_{t_1,t_2}$, there is nothing to check here.

\item Let us check that $(g,g)v_{t_1,t_2}=v_{t_1,t_2}$ for any $g\in\GLB$. Indeed, since $\pi_{t_1,t_2}$ is unitary
representation,
\begin{multline*}
 ((g,g)v_{t_1,t_2},f)=(v_{t_1,t_2},(g^{-1},g^{-1})f)\\=\int_{Gr(V)} f(gX,gX)
 \sqrt{\frac{(g,g)\rho_{t_1,t_2}}{\rho_{t_1,t_2}}(X,X)} \eta_{t_1,t_2}(dX)\\ \stackrel{Y=gX}{=} \int_{Gr(V)} f(Y,Y)
 \eta_{t_1,t_2}(dY)=(v_{t_1,t_2},f)
\end{multline*}
\item We have already shown that $I_\B v_{t_1,t_2}= v^0_{t_1,t_2}$. Then
$$
 (I_\B v_{t_1,t_2},v_{t_1,t_2})=\int_{Gr(V)} \eta_{t_1,t_2}(dX) = 1.
$$
\end{enumerate}
Now we compute the  trace of this representation.

For $g\in\GL{n}$ we have
\begin{multline*}
 (\pi_{t_1,t_2}(I_g,e)v_{t_1,t_2},v_{t_1,t_2})= (\pi_{t_1,t_2}(g\cdot I_{e(n)},e)v_{t_1,t_2},v_{t_1,t_2})=(\pi_{t_1,t_2}(g,e)
 v^n,v)\\=\mu(\BI_n)\sum_{X\in Gr(V^n)}(\pi_{t_1,t_2}(g,e)q^{-\dim(X)(\dim(X)+1)/2+\sum_k k (Sym(X)_k)}\mathcal U^n(X,X),v)
\end{multline*}
By the definition for $X,Y\in Gr(V_n)$ we have
$$
 (\mathcal U^n(X,Y),v)=\begin{cases} q^{m(m+1)/2-\sum_k k (Sym(X)_k)} t_1^m t_2^{n-m} ,\quad X=Y,
 \\ 0,\text{ otherwise,}\end{cases}
$$
where $m=\dim(X)$. It follows that
\begin{equation}
\label{eq_x2}
 (\pi_{t_1,t_2}(I_g,e)v_{t_1,t_2},v_{t_1,t_2}) =\mu(\BI_n)\sum_{X\in Gr(V^n): gX=X}  t_1^{\dim(X)} t_2^{n-\dim(X)}.
\end{equation}
Our next aim is to decompose the function $\chi(g):=(\pi_{t_1,t_2}(I_g,e)v,v)$ into the sum of
matrix traces of irreducible representations of $\GL{n}$. We rewrite \eqref{eq_x2}as
$$
 \chi(g)=\sum_{m=0}^n t_1^n t_2^{n-m} \psi_m(g)
 $$
with
$$
 \psi_m(g)=\mu(\BI_n)\#\{X\in Gr(V^n): \dim(X)=m,\, gX=X\}
$$
Let $\Psi_m$ be the natural representation of $\GL{n}$ in the space of functions on the set of all
subspaces of $V_n$ of dimension $m$. If we view $\Psi_m$ as a representation of $\A(\GLB)_n$, then
its matrix trace $Trace(\Psi_m(I_g)$ is precisely $\psi_m(g)$; the prefactor $\mu(\BI_n)$ arises
from the identification of $\A(\GLB)$ and the group algebra of $\GL{n}$ (see Proposition
\ref{prop_A_n_as_group_algebra}).

The decomposition of $\Psi_m$ into irreducible representations is well known (see e.g.\
\cite{Steinberg}). We have
\begin{equation}
\label{eq_flag_rep_decomp}
 \Psi_m=\bigoplus_{\lambda} K_{(n-m,m),\lambda} \psi^\lambda,
\end{equation}
where $\psi^\lambda$ is the irreducible unipotent representation of $\GL{n}$ indexed by the Young
diagram with $n$ boxes $\lambda$ and $K_{(n-m,m),\lambda}$ is the \emph{Kostka number}. These
numbers do not depend on $q$ and coincide with similar coefficients for the decomposition of the
representation of symmetric group $\Sym(n)$ in the space of functions on the set of all
$m$-element subsets of the set $\{1,2,\dots,n\}$. It is convenient for us to use yet another
definition related to the symmetric functions:
$$
 h_m h_{n-m} =\sum  K_{(n-m,m),\lambda} s_\lambda,
$$
where $h_m$ is the complete symmetric function and $s_\lambda$ is the Schur function. The last
formula can be shown to be equivalent to the definition of Kostka numbers through
\eqref{eq_Kostka}.

\eqref{eq_flag_rep_decomp} implies that
$$
 \chi=\sum_\lambda \left(\sum_m K_{(n-m,m),\lambda} t_1^{m}t_2^{n-m}\right) \chi^\lambda,
$$
where $\chi^\lambda$ is the conventional character (matrix trace) of the unipotent representation
indexed by the Young diagram with $n$ boxes $\lambda$.

Next, observe that
\begin{multline*}
\sum_m K_{(n-m,m),\lambda} t_1^{m}t_2^{n-m}=\sum_{m\le n/2} K_{(n-m,m),\lambda}
m_{(n-m,m)}(t_1,t_2)= s_\lambda (t_1,t_2),
\end{multline*}
where $m_\mu$ is the monomial symmetric function indexed by the Young diagram $\mu$ and for a
symmetric function $f(x_1,x_2,\dots)$ the notation $f(t_1,t_2)$ means the specialization
$f(t_1,t_2,0,0,\dots)$.

We arrive at the final formula
$$
 \chi=\sum_\lambda s_\lambda(t_1,t_2) \chi^\lambda,
$$
which coincides with the decomposition of the extreme unipotent  trace of $\A(\GLB)$ with
 parameters $\alpha_1=t_1$, $\alpha_2=t_2$ given in Theorem \ref{theorem_characters_of_GLB}.
\end{proof}

\subsection{Representations related to spaces of flags}
\label{subsect_flags}

The results of Section \ref{Subsect_grassman} can be generalized to give a construction for the
representations of $\GLB$ corresponding to the extreme unipotent representation with arbitrary
sequence of parameters $\alpha_i$.

Suppose that we have $r$ non-zero parameters $\alpha_i$:
$$
 \alpha_1=t_1,\, \alpha_2=t_2,\dots,\, \alpha_r=t_r,
$$
with $r$ being either finite or $r=+\infty$.

Let $Fl_r(V)$ denote the space of all length $r-1$ \emph{decreasing} flags in $V$, i.e.\
$$
 Fl_r(V)=\{X_1\supseteq X_2\supseteq \dots\supseteq X_{r-1}\mid X_i\in Gr(V)\}.
$$
In particular, $Fl_2(V)=Gr(V)$. Note that, in principle, we allow non-strict inclusions  in the
above definition, e.g.\ $X_1$ might be equal to $X_2$. However, with respect to the measures we
use, the inclusions turn out to be almost surely strict. If $r=\infty$, then we also demand that
$\bigcap_i X_i =\{0\}$.

 The group $\GLB$ naturally acts in $Fl_r(V)$ and, similarly
to the grassmanian case, we define:
$$
 Fl_r^2(V)=\{(F,H)\in Fl_r(V)\times Fl_r(V)\mid \exists g\in\GLB: gF=H\}.
$$

For a flag $F\in Fl_r(V)$ let $F^{(i)}$, $i=1,\dots,r-1$ denote its subspaces, i.e.\
$F=F^{(1)}\supseteq\dots\supseteq F^{(r-1)}$. The symbol $Sym(F)$ of the flag $F\in Fl_k(V)$ is
defined as the coordinate-wise sum of the symbols of $F^{(i)}$:
$$
 Sym(F)=\left(\sum_{i=1}^{r-1} Sym(F^{i})_1, \sum_{i=1}^{r-1} Sym(F^{i})_2,\dots\right).
$$
Note that this sum is well-defined even for $r=\infty$ as follows from the condition $\bigcap_i
X_i =\emptyset$.

Let $\mathcal N_r$ denote the set $\{0,1\dots,r-1\}$.  For a sequence $f\in \mathcal N_r^{\infty}$
let $Sch(f)$ denote the set of all flags in $Fl_r(V)$ with symbol $f$. $Sch(f)$ has a
distinguished coordinate flag, which we denote $C(f)$. $Sch(f)$ is a $\B$-orbit and, thus, has a
unique $\B$-invariant probability measure.

Next, we define the map $\phi_r$:
$$
 \phi_r: \mathcal N_r^{\infty} \times \B \to Fl_r(V),\quad \phi_r(x,g)=g C(x).
$$
Let $\eta_{t_1,\dots,t_r}$ be the $\phi_r$-pushforward of the product of Bernoulli measure $P$ on
$\mathcal N_r^\infty$ with $Prob(k)=t_k$ and Haar measure on $\B$.

Let $\mathcal T_r$ denote the set of pairs of sequences $(x,y)\in \mathcal N_r^{\infty}\times
\mathcal N_r^{\infty}$ such that $x$ is a finite permutation of $y$.

$\mathcal T_r$ is equipped with sigma algebra spanned by the sets
$A_{i_1,\dots,i_n}^{j_1,\dots,j_n}$, where $i_1,\dots, i_n$ and $j_1,\dots, j_n$ are two sequences
from $\mathcal N_r^n$ which are permutations of each other
$$
 A_{i_1,\dots,i_n}^{j_1,\dots,j_n}=\{(x,y)\in \mathcal T_r\mid x_1=i_1,\dots, x_n=i_n,\, y_1=j_1,\dots,
 y_n= j_n,\, y_k=x_k, \text{ for } k>n\}.
$$

Define the measure $R_{t_1,\dots,t_r}$ on $\mathcal T_r$ setting
$$
 R_{t_1,\dots,t_r}(A_{i_1,\dots,i_n}^{j_1,\dots,j_n})=\prod_{\ell=1}^n t_{i_\ell}.
$$

Let $\psi_r$ be the map
$$
 \psi_r: \mathcal T_r \times \B\times \B \to Fl_r^2(V),\quad \phi_r((x,y),g,h)=(g C(x), hC(y)).
$$
We define the measure $\rho_{t_1,\dots,t_r}$ on $Fl_r(V)$ as the $\psi_r$-pushforward of the
measure $R_{t_1,\dots,t_r}\otimes \mu_\B \otimes \mu_\B$.

Similarly to Proposition \ref{prop_eta_quasiinvar} and Corollary \ref{corol_rho_quasiinvar} one
proves that $\eta_{t_1,\dots,t_r}$ is $\GLB$-quasiinvariant and $\rho_{t_1,\dots,t_r}$ is
$\GLB\times\GLB$--quasiinvariant.

Therefore there is a natural unitary representation $\pi_{t_1,\dots,t_r}$ of $\GLB\times\GLB$ in
$L_2(Fl_r(V),\rho_{t_1,\dots,t_r})$.

Similarly, to $Gr^2(V)$ we define a topological structure of $Fl_r(V)$ and consider the space
$C^0(Fl_r^2(V))$ of continuous functions with compact support. Let $v_{t_1,\dots,t_r}\in
C^0\big(Fl_r^2(V)\big)^*$ denote the following linear functional
$$
 v_{t_1,\dots,t_r}:C^0(Fl^2_r(V))\to\mathbb C,\, v_{t_1,\dots,t_r}(f)=\int_{Fl_r(V)} f(X,X) \eta_{t_1,\dots,t_r}(dX).
$$
We further set $\hat \pi_{t_1,\dots,t_r}$ to be the restriction of $\pi_{t_1,\dots,t_r}$ on the
$L_2$--closure of the intersection of $C^0(Fl_r^2(V))$ and cyclic span $Sp(v_{t_1,\dots,t_r})$ of
$v_{t_1,\dots,t_r}$.

\begin{theorem}
\label{theorem_construction_flags}
 The triplet $(\hat \pi_{t_1,\dots,t_r}, C^0(Fl^2_r(V))\cap Sp(v_{t_1,\dots,t_r}), v_{t_1,\dots,t_r})$ is a generalized spherical
 representation. Its spherical function gives the extreme unipotent  trace of $\A(\GLB)$ with
 parameters $\alpha_1=t_1$,\dots, $\alpha_r=t_r$. The restriction of this representations on the first
 component is von Neumann factor representation of type $II_\infty$.
\end{theorem}
The proof repeats that of Theorem \ref{Theorem_construction_gras}.

\section{Biregular representation of $\GLB$}
\label{Section_harmonic}

Recall that $\GLB$ is a locally compact group with biinvariant Haar measure $\mu_{\GLB}$ and
consider the Hilbert space $H=L_2(\GLB,\mu_{\GLB})$. Let $\pi_{Reg}$ denote the natural
representation of $(\GLB\times\GLB)$ in $H$ by left and right translations. Let $H_1\subset H$ be
the subspace $C[\GLB]$ of all continuous functions on $\GLB$ and let $\delta_e$ denote the
delta-function at the identity element of $\GLB$:
$$
 \delta_e(f)=f(e), \quad f\in H_1.
$$
\begin{theorem}[On the structure of biregular representation]
\label{Theorem_biregular}
 The triplet $(\pi_{Reg},C[\GLB],\delta_e)$ is a generalized spherical representation of $\GLB$. Its
 spherical function $\chi$ has the following decomposition into extreme traces of $\A(\GLB)$:
 $$
  \chi=\sum_{f\in\CY'} C(f)\chi^{0,1^{(q)},f},
 $$
 where $1^{(q)}$ means the geometric series $((1-q^{-1}), (1-q^{-1})q^{-1},
 (1-q^{-1})q^{-2},\dots)$ and
  $$
  C(f)= (q-1)^{|f|}\prod_{c\in\C_d} \frac{q^{d n(s(c))}} {\prod_{\square\in s(c)} (q^{d h(\square)}-1)}.
  $$
\end{theorem}
\begin{proof}
  Observe that $\pi_{Reg}(I_{e(n)},e)\delta_e=I_{e(n)}$. Therefore, for $g\in \GL{n}$ we have
  $$
  \pi_{Reg}(I_{g},e)\delta_e=\begin{cases} 1, g=e(n),\\ 0,\text{ otherwise.} \end{cases}
  $$
  It follows that the restriction of $\chi$ to $\A(\GLB)_n$ (under the identification $\A(\GLB)_n\simeq \mathbb C(\GL{n})$
  is the character of the regular representation of $\GL{n}$ multiplied by the constant
  $$
  \frac{(q-1)^n q^{n(n-1)/2}}{|\GL{n}|}=\prod_{i=1}^n \frac{q-1}{q^i-1}.
  $$
  Using the well-known decomposition of the regular representation of a finite group into irreducibles we get
  $$
   \chi \rule[-2.5mm]{.4pt}{5mm}_{\,\A(\GLB)_n}=\prod_{i=1}^n \frac{q-1}{q^i-1} \cdot \sum_{s\in \CY_n} {\rm dim}_q(s) \chi^s,
  $$
  where ${\rm dim}_q(s)$ is the dimension of the irreducible representation of $\GL{n}$ indexed by
  $s$ which can be computed using the following $q$-analogue of the hook formula:
  $$
   {\dim}_q(s)= (q^n-1)\dots(q-1)\prod_{d\ge 1} \prod_{c\in\C_d} \frac{q^{d n(s(c))}} {\prod_{\square\in s(c)} (q^{d
   h(\square)}-1)}.
  $$
  Extracting the factor with $c=``x-1"$ and applying the identity
  (here we used \cite[Chapter I, Section 3, (3.8)]{M} and \cite[Chapter I, Section 3, Exercise 2]{M})
 \begin{multline*}
 \Sp_{0,1^{(q)},1}[s_\lambda]=
  s_{\lambda'}\biggl(1-q^{-1}, q^{-1}(1-q^{-1}),q^{-2}(1-q^{-1}),\dots\biggr)\\=(q-1)^{|\lambda|}\frac{q^{n(\lambda)}}
  {\prod_{\square\in \lambda} (q^{h(\square)}-1)}
  \end{multline*}
  we arrive at
  \begin{equation}
  \label{eq_x10}
   \chi \rule[-2.5mm]{.4pt}{5mm}_{\,\A(\GLB)_n}=  \sum_{f\in \CY'} C(f) \sum_{\lambda\in \Y_{n-|f|}}
   \chi^{f+E_1(\lambda)} \Sp_{0,1^{(q)},1}[s_\lambda],
  \end{equation}
  with
  $$
  C(f)= (q-1)^{|f|}\prod_{c\in\C_d} \frac{q^{d n(s(c))}} {\prod_{\square\in s(c)} (q^{d h(\square)}-1)}.
  $$
  It remains to compare \eqref{eq_x10} with \eqref{eq_extreme_GLB}.
\end{proof}

\section{Appendix: $\GLU$}

There is another distinguished infinite-dimensional group over finite field, which is a group of
almost uni-uppertriangular matrices.

\begin{definition} $\GLU$ is a subgroup of $\GLB$ defined through
$$
 \GLU=\{[X_{ij}]\in \GLB : X_{ii}=1\text{ for large enough }i\}.
$$
\end{definition}

The whole theory for $\GLU$ is very much parallel to that of $\GLB$.

The group $\GLU$ is an inductive limit of groups $\GLU_n$: $\GLU=\bigcup_{n=0}^{\infty} \GLU_n$.
$$
 \GLU_n=\{ [X_{ij}]\in\GLU\mid X_{ij}=0\text{ if both } i>j\text{ and }i>n; \quad X_{ii}=1 \text{ for }i>n\},
$$
in particular, $\GLU_0=\U\subset \GLU$ is the subgroup of all unipotent upper-triangular matrices.

Each $\GLU_n$ is compact group (with topology of pointwise convergence of matrix elements). $\GLU$
as an inductive limit of $\GLU_n$ is a locally compact topological group. Let $\mu^\GLU$ denote
the biinvariant Haar measure on $\GLU$ normalized by the condition $\mu^{\GLU}(\U)=1$.

The space $L_1(\GLU,\mu^\GLU)$ is a Banach involutive algebra with multiplication given by the
convolution.

\begin{definition}
$\A(\GLU)\subset L_1(\GLU,\mu^\GLU)$ is defined as the subalgebra formed by all locally constant
functions with compact support. In other words, a function $f(X)$ belongs to $\A(\GLU)$ if their
exists $n$ and a function $f_n: \GL{n}\to\mathbb C $ such that:
$$
f(X)=\begin{cases} f_n(X^{(n)}),\text{ if } X\in\GLU_n, \\0,\text{ otherwise.}
\end{cases}
$$
\end{definition}
Clearly, $\A(\GLU)$ is dense in $L_1(\GLU,\mu^\GLU)$. Note that $\A(\GLU)$ does not have a unit
element.

As and above, we call a (linear) function $\chi: \A(\GLU) \to \mathbb C$ a trace of $\A(\GLU)$ if
\begin{enumerate}
 \item $\chi$ is central, i.e.\ $\chi(WU)=\chi(UW)$,
 \item $\chi$ is positive definite, i.e.\ $\chi(W^*W)\ge 0$ for any $W\in \A(\GLU)$,
\end{enumerate}

\noindent {\bf Remark. } It is impossible to normalize the  traces, i.e.\ for any $a\in\A(\GLU)$
there exists a  trace $\chi$ such that $\chi(a)=0$.

\bigskip

For any matrix $g\in\GL{n}$ let $I^\GLU_g\in\A(\GLU)$ denote the function
$$
 I^\GLU_g(X)=
 \begin{cases}
  1, \text{ if } X\in\GLU_n \text{ and }X^{(n)}=g, \\
  0, \text{ otherwise.}
 \end{cases}
$$
Let $e(n)$ denote the identity element of $\GL{n}$. Then
$$I^{\GLU}_g(X)=I^{\GLU}_{e(n)}(Xg^{-1})=g\cdot I^{\GLU}_{e(n)}.$$

Denote $\A(\GLU)_n=<I^\GLU_g\mid g\in\GL{n}>$. The following proposition is straightforward
\begin{proposition}
\label{prop_A_n_as_group_algebra_GLU}
 $\A(\GLU)_n$ is a subalgebra of $\A(\GLU)$ isomorphic to the conventional group algebra $\mathbb
 C(\GL{n})$. If $e_g$ denotes the natural basis of $\mathbb
 C(\GL{n})$ then the isomorphism is given
$e_g \to q^{n(n-1)/2} I^\GLU_g$.
\end{proposition}

Observe that $\A(\GLU)_n\subset \A(\GLU)_{n+1}$. In the basis $I^\GLU_g$ this inclusion is given
by
$$
 i_n: I^\GLU_g \to \sum_{h\in Ext^\GLU(g)} I^\GLU_h,
$$
where for $g\in\GL{n}$ we have
\begin{multline*}
 Ext^\GLU(g)=\biggl\{[h_{ij}]\in\GL{n+1}\mid\\ h^{(n)}=g\text{ and }
 h_{n,1}=h_{n,2}=\dots={h_{n,n-1}}=0, \, h_{n,n}=1\biggr\}.
\end{multline*}

The algebra $\A(\GLU)$ can be identified with the inductive limit of algebras $\A(\GLU)_n$:
$$
 \A(\GLU)=\ilim_{n\to\infty} \A(\GLU)_n = \bigcup_{n} \A(\GLU)_n.
$$
Thus, $\A(\GLU)$ is a \emph{locally semisimple} algebra.

\begin{proposition}
\label{prop_decomposition_of_restriction_GLU}
 The set of all  traces of $\A(\GLU)_n$ is a simplicial cone spanned by  traces
 $\chi^f$, $f\in\CY_n$. In other words, if $\chi^n$ is a  trace of $\A(\GLU)_n$, then there exist unique
 nonnegative coefficients $c(f)$ such that
 $$
  \chi^n(\cdot)=\sum_{f\in\CY_n} c(f) \chi^f(\cdot).
 $$
\end{proposition}
\begin{proof} The proof repeats that of Proposition \ref{cor_decomposition_of_restriction}
\end{proof}

\begin{definition}
For two families $f\in\CY_n$ and $g\in\CY_{n-1}$ we say that $g$ precedes $f$ and write
$g\prec_{\GLU} f$ if there exists $a\in\F^*$ for which
\begin{enumerate}
\item $f(``x-a")\setminus g(``x-a")$ is one box,
\item $f(u)=g(u)$ for all $u\ne ``x-a"$.
\end{enumerate}
\end{definition}
Note that this definition is \emph{different} from that of $\prec_\GLB$.

Similarly to Theorem \ref{proposition_branching_of_characters_GLB} one proves the following
statement.
\begin{proposition}
\label{proposition_branching_of_characters_GLU}
 Let $\pi^f$  be the irreducible representation of algebra
 $\A(\GLU)_n$ (equivalently, of the group $\GL{n}$) parameterized by $f\in\CY_n$ and let
$\chi^f$ be its conventional character (i.e.\ matrix trace). The restrictions of $\pi^f$ and
$\chi^f$ to the subalgebra
 $\A(\GLU)_{n-1}$ admit the following decomposition:
$$
\chi^f\rule[-2.5mm]{.4pt}{5mm}_{\,\A(\GLU)_{n-1}}=\sum_{g\prec_{\GLU} f} \chi^g,
$$
equivalently,
$$
 \pi^f \rule[-2.5mm]{.4pt}{5mm}_{\,\A(\GLU)_{n-1}} = \mathcal N \oplus \bigoplus_{g\prec_{\GLU} f} \chi^g,
$$
where $\mathcal N$ is a zero representations of $\A(\GLU)_{n-1}$ of dimension
$\dim(f)-\sum_{g\prec_{\GLU} f} \dim(g)$.
\end{proposition}

We need to introduce some additional notations to state an analogue of Theorem
\ref{theorem_characters_of_GLB} for $\GLU$.

\begin{definition}
$\CY`\subset \CY$ is defined as the set of families $f$ of Young diagrams such that
$f(``x-a")=\emptyset$ for $a\in\F^*$.
\end{definition}

\begin{definition}
 $\Omega(\GLU)$ is the set of quadruples $(\alpha,\beta,\gamma,f)$, where $\gamma=\{\gamma^j\}$,
$j\in\F^*$, $\alpha=\{\alpha^{j}_i\}$,
 $\beta=\{\beta^j_i\}$, $i=1,2,3\dots$, $j\in\F^*$; for every $j$ the sequences $\alpha^{j}_i$ and $\beta^{j}_i$
 and the number $\gamma^j$ satisfy \eqref{eq_summable_sequence}; moreover
$ \sum_{j} \gamma^j= 1$ and  $f\in\CY`$.
\end{definition}

\begin{theorem}[Classification theorem for finite traces of $\A(\GLU)$]
\label{theorem_characters_of_GLU}
 The extreme rays of the set of  traces of $\A(\GLU)$ are parameterized by elements of $\Omega(\GLU)$. For
 $\omega=(\alpha,\beta,\gamma, f)\in \Omega(\GLU)$ the corresponding ray is $\mathbb R_+ \chi^{\omega}(\cdot)$ and for $g\in\GL{n}$ we have
  $  \chi^{\omega}(I^\GLU_g)=0
 $
 if $n<|f|$, otherwise,
 \begin{equation}
 \label{eq_extreme_GLU}
  \chi^{\omega}(I^\GLU_g)=\sum_{\lambda^j\in \Y:\, \sum |\lambda^j|=n-|f|}
  \chi^{f+ \sum_j E_j(\lambda^j)}(I^\GLU_g) \prod_{j\in\F^*} \Sp_{\alpha^j,\beta^j,\gamma^j} [s_{\lambda^j}],
 \end{equation}
  Where $E_j(\lambda)\in\CY$ is a family taking value $\lambda$ in $``x-j"$ $(j\in\F^*)$ and taking
 $\emptyset$ in all other points.
\end{theorem}

\begin{proof}
 The argument starts similarly to that of Theorem \ref{theorem_characters_of_GLB}.
 For a family $f\in\CY`$ let $\CY^{[f]}\subset\CY$ denote the set of families $h\in\CY$ such that
 $h(u)=f(u)$ for all $u\in\mathcal C\setminus \bigcup_{j\in\F^*}\{``x-j"\}$.

 Moreover, for a family $f\in\CY`$ let $\Theta^f$ denote the convex cone of  traces $\chi$ of $\A(\GLU)$ such that
 such that for $n<|f|$ the restriction $\chi\rule[-2.3mm]{.4pt}{4mm}_{\,\A(\GLU)_{n-1}}$ vanishes and for $n\ge |f|$
 in the decomposition
$$
 \chi \rule[-2.3mm]{.4pt}{4mm}_{\,\A(\GLU)_{n}} =\sum_{h\in\CY_n} c(h) \chi^h(\cdot).
$$
$c(h)=0$ unless $h\in\CY^{[f]}$. Let $\Theta^\emptyset$ denote the set $\Theta^f$ for $f$ being
the empty family. Similarly to the proof of Theorem \ref{theorem_characters_of_GLB} for any
$f\in\CY`$ the convex cone $\Theta^f$ is affine isomorphic to $\Theta^\emptyset$ and the statement
of Theorem \ref{theorem_characters_of_GLU} is reduced to the identification of all extreme rays of
$\Theta^\emptyset$. The rest of the proof is this identification.

Take $q-1$ countable sets of variables $(x^j_i)_{i=1,2,\dots}$, $j\in\F^*$ and let
$\Lambda^{\otimes (q-1)}$ denote the algebra of (polynomial) functions symmetric in variables
$x^j_i$, $i=1,2,\dots$ for every fixed $j\in\F^*$. Let $\Lambda^j$ denote the subalgebra of
symmetric functions in $x^j_1, x^j_2,\dots$. For any $j$ and any symmetric function $r\in\Lambda$
let $r(x^j)$ denote the corresponding symmetric function in variables $x^j_1, x^j_2,\dots$. In
particular, $p_k(x^j)\in \Lambda^j$ are the Newton power sum in variables $x^j$
$$
 p_k(x^j)=\sum_{i=1}^{\infty} (x^j_i)^k.
$$
Clearly, the functions
$$
 \prod_{j\in \F^*} s_{\lambda^j}(x^j),\quad \lambda^j\in\Y
$$
form a linear basis in $\Lambda^{\otimes (q-1)}$.

Let $\Delta$ denote the cone of linear functionals $w$ on $\Lambda^{\otimes (q-1)}$ satisfying:
$$
 w: \Lambda^{\otimes (q-1)}\to\mathbb C
$$
\begin{enumerate}
\item
 $w\left[\prod_{j\in\F^*} s_{\lambda^j}(x^j)\right]\ge 0$ for any
    $q-1$ Young diagrams $\{\lambda^j\}$.
\item
$w\left[u(\sum_{j\in \F^*} p_1(x^j))\right]=w[u]$, for any $u\in \Lambda^{\otimes (q-1)}$.
\end{enumerate}

We claim that $\Delta$ and $\Theta^\emptyset$ are affine isomorphic. Under this identification a
 trace $\chi\in\Theta^\emptyset$ corresponds to a functional $w_\chi$ such that
$$
 w_\chi\left[\prod_{j\in\F^*} s_{h(``x-j")}(x^j)\right]=c(h)
$$
for any $h\in\CY^{[\emptyset]}$ and for $h\in \CY^{[\emptyset]}\bigcap \CY_n$ the number $c(h)$ is
defined as the coefficient in the decomposition
$$
 \chi \rule[-2.3mm]{.4pt}{4mm}_{\,\A(\GLU)_{n}} =\sum_{h\in\CY_n} c(h) \chi^h(\cdot).
$$
Indeed, condition 1 translates into the non-negativity of the coefficients $c(h)$ and condition 2
translates into the statement of Proposition \ref{proposition_branching_of_characters_GLU}.

Now we can again use the \emph{Ring Theorem} (see \cite{VK_ring}, \cite{VK_Long}, \cite{Kerov_book} and also
\cite[Section 8.7]{GO_Ring}) for studying the structure of the set $\Delta$. This theorem yields
that $h\in\Delta$ is an element of an extreme ray (in other words, $h$ is indecomposable) if and
only if $h=r \hat h$, where $r\in\mathbb R_+$ and $\hat h\in\Delta$ is a multiplicative
functional, i.e.\ $\hat h(uv)=\hat h(u)\hat h(v)$.

Now let $h\in\Delta$ be a multiplicative functional and let $h^j$, $j\in\F^*$ be its restrictions
on $\Lambda^j$. Clearly,
\begin{equation}
\label{eq_x5}
 h\left[\prod_{j\in\F^*} s_{\lambda^j}(x^j)\right]=\prod_{j\in\F^*} h^j[s_{\lambda^j}(x^j)]
\end{equation}
Define a new linear functional $\widehat h^j$ on $\Lambda^j$ through
$$
\widehat h^j[u]:=\frac{h^j[u]}{\left(h^j[p_1(x^j)]\right)^{deg(u)}},
$$
where $deg(u)$ is the degree of polynomial $u$. The functional $\widehat h^j$ on $\Lambda^j$ is
multiplicative and satisfies
\begin{enumerate}
\item $\widehat h^j [s_\lambda(x^j)]\ge 0$ for every $\lambda\in \Y$,
\item $\widehat h^j [u p_1(x_j)]=h^j[u]$,
\item $h^j[p_1 (x_j)]=1$.
\end{enumerate}
Such functional are classified by Thoma's theorem, they correspond to extreme points of the set of
normalized characters of $S(\infty)$, see \cite{VK_Long}, \cite{Kerov_book}. They are
parameterized by sequences $\alpha,\beta$, satisfying \eqref{eq_summable_sequence} with
$\gamma=1$. We have:
$$
 \widehat h^j[s_\lambda(x^j)]= \Sp_{\alpha,\beta,1} [s_\lambda].
$$
 Now set $\gamma_j=h^j[p_1(x^j)]$. Then we have
$$
 h^j[s_\lambda(x^j)]= \Sp_{\alpha^j,\beta^j,\gamma^j}[s_\lambda]
$$
for certain $\alpha^j,\beta^j,\gamma^j$. Substituting into \eqref{eq_x5} we arrive at the desired
statement.



\end{proof}

\bigskip

 Theorem \ref{theorem_characters_of_GLU} yields that for $\GLU$ and
$\omega=(\alpha,\beta,f)\in\Omega({\GLU})$ the  trace $\chi^\omega$ is unipotent if
$f\equiv\emptyset$ and $\alpha^j_i=\beta^j_i=0$ for $j\ne 1$. Note that if we identify
$\A({\GLB})_n\simeq\mathbb C(\GL{n})\simeq \A({\GLU})_n$, then unipotent  traces of $\A(\GLB)$ and
$\A(\GLU)$ are \emph{the same} functions.

\end{document}